\documentclass[12pt]{amsart}

\usepackage{graphicx}
\usepackage{amsfonts}
\usepackage{epsf}
\usepackage{amssymb}
\usepackage{amsmath}
\usepackage{amscd}
\usepackage{hyperref}
\usepackage{color}
\usepackage{bbm}

\newtheorem{theorem}{Theorem}[section]
\newtheorem{proposition}[theorem]{Proposition}
\newtheorem{lemma}[theorem]{Lemma}
\newtheorem{question}[theorem]{Question}
\newtheorem{corollary}[theorem]{Corollary} 
\newtheorem{conjecture}[theorem]{Conjecture}

\theoremstyle{remark}
\newtheorem{definition}[theorem]{Definition}
\newtheorem{example}[theorem]{Example}
\newtheorem{remark}{Remark}[section]

\newcommand{\eps}{\epsilon}

\input xy
\xyoption{all}

\begin{document}

\title{The concordance crosscap number and rational Witt span of a knot}
\begin{abstract}
The concordance crosscap number $\gamma_c(K)$ of a knot $K$ is the smallest crosscap number $\gamma_3(K')$ of any knot $K'$ concordant to $K$ (and with $\gamma_3(K')$ defined as the least first Betti number of any nonorientable surface $\Sigma$ embedded in $S^3$ with boundary $K'$). This invariant has been introduced and studied by Zhang \cite{Zhang} using knot determinants and signatures, and has further been studied by Livingston \cite{Livingston} using the Alexander polynomial. We show in this work that the rational Witt class of a knot can be used to obtain a lower bound on the concordance crosscap number, by means of a new integer-valued invariant we call the {\em Witt span} of the knot.  
\end{abstract}
%
%
\thanks{The author was partially supported by the Simons Foundation, Award ID 524394, and by the NSF, Grant No. DMS--1906413. }
\author{Stanislav Jabuka}
\email{jabuka@unr.edu}
\address{Department of Mathematics and Statistics, University of Nevada, Reno NV 89557, USA.}
\maketitle
\section{Introduction}
\subsection{Background} \label{Introduction}
The study of knots by means of the surfaces they bound goes back to at least 1930 when Frankl and Pontrjagin \cite{FranklPontrjagin} proved the existence of a Seifert surface for any knot, something that was expounded on just a few years later by Seifert in his influential work \cite{Seifert}. The relationship between knots and surfaces naturally leads to the definition of several flavors of knot genera. For this work three such invariants will be relevant. The {\em crosscap number} or {\em nonorientable 3-genus} $\gamma_3(K)$ of a knot $K$ is the smallest first Betti number of any nonorientable surface $\Sigma$ embedded in $S^3$ with $\partial \Sigma =K$. It was defined by Clark \cite{Clark} in 1976, a paper in which he proved this useful theorem:
\begin{theorem}[Clark \cite{Clark}] \label{Clark's Theorem}
A knot $K$ has crosscap number equal to 1 if and only if $K$ is the $(2,n)$-cable of some knot. 
\end{theorem}  
A 4-dimensional analogue of the crosscap number was defined in 2000 by Murakami and Yasuhara \cite{MurakamiYasuhara}. Called the {\em 4-dimensional crosscap number} or {\em the nonorientable 4-genus} and denoted by $\gamma_4(K)$, it is the smallest first Betti number of any nonorientable surface $\Sigma$, smoothly and properly embedded into the 4-ball $D^4$ and with $\partial \Sigma =K$.

The {\em concordance crosscap number} or {\em the nonorientable concordance 3-genus} of $K$, denoted $\gamma_c(K)$, is defined as 
$$\gamma_c(K) = \min \{\gamma_3(K')\, |\, K' \text{ is smoothly concordant to } K\}.$$
It can be seen as an intermediary between the 3- and 4-dimensional definitions of $\gamma_3$ and $\gamma_4$ respectively, and its value clearly lies between the two:
\begin{equation} \label{Double Inequality}
\gamma_4(K) \le \gamma_c(K) \le \gamma_3(K). 
\end{equation}
Batson's landmark paper \cite{Batson} showed that the nonorientable 4-genus $\gamma_4$ can grow arbitrarily large, and by implication so can $\gamma_c$. That said, $\gamma_4$ remains a notoriously difficult to compute, and Batson's lower bound for $\gamma_4(K)$ (as well as a second lower bound for $\gamma_4(K)$ by Ozsv\'ath, Stipsicz and Szab\'o \cite{OSS}) vanishes for alternating knots. For these reasons, it is of independent interest to develop lower bounds on $\gamma_c(K)$, ones that are computationally accessible and whose applicability is not limited to particular families of knots. This is the goal of this paper.

The concordance crosscap number was first defined by Zhang \cite{Zhang} in 2007. In her work she identifies an infinite family of 3-stranded pretzel knots for which $\gamma_4$ equals 1 and for which $\gamma_c$ is greater than 1. Her obstruction to $\gamma_c(K)=1$ relies on Clark's Theorem \ref{Clark's Theorem}, and a clever use of the  signature and determinant of the knot. 

The only other work published to date on the concordance crosscap number is by Livingston \cite{Livingston}. He also uses Clark's Theorem \ref{Clark's Theorem} along with the Alexander polynomial to prove this result:
\begin{theorem}[Livingston \cite{Livingston}] \label{Livingston's Theorem}
Suppose $\gamma_c(K)=1$ and set $q = |\sigma(K)|+1$. For all odd prime power divisors $p$ of $q$, the $2p$–cyclotomic polynomial $\phi_{2p}(t)$ has odd exponent in the Alexander polynomial $\Delta_K(t)$. Furthermore, every other symmetric irreducible polynomial $\delta(t)$ with odd exponent
in $\Delta_K(t)$ satisfies $\delta(-1) = \pm 1$. 
\end{theorem}

\subsection{Results} The goal of this work is to introduce and study a new concordance invariant of knots, called the {\em rational Witt span}, or {\em Witt span} for short, and use it to bound from below the concordance crosscap number $\gamma_c(K)$. To define this new invariant, recall that the Witt ring of the rational numbers, denoted $W(\mathbb Q)$, is a commutative ring that fits into this split exact sequence of Abelian groups: 
\begin{equation} \label{SESForW(Q)} 
0 \to \mathbb Z \stackrel{\iota}{\longrightarrow} W(\mathbb Q) \stackrel{\partial}{\longrightarrow} \bigoplus_{p \text{ prime}} W(\mathbb F_p) \to 0.
\end{equation}
Here $\mathbb F_p$ is the field of $p$ elements, and $W(\mathbb F_p)$ is its Witt ring. The map $\partial :W(\mathbb Q)\to \oplus _{p \text{ prime}} W(\mathbb F_p)$ is the direct sum $\partial = \oplus _p \partial _p$ of group homomorphisms $\partial _p:W(\mathbb Q) \to W(\mathbb F_p)$. More generally, for any field $\mathbb F$ the elements of its Witt ring $W(\mathbb F)$ are equivalence classes of finite rank, non-degenerate, symmetric, bilinear forms over $\mathbb F$. This allows one to define the Witt class $W(K)$ of a knot $K$ to be the equivalence class in $W(\mathbb Q)$ of the bilinear form on $H_1(\Sigma ; \mathbb Q)$ coming from the symmetric linking form $V+V^\tau$ associated to any Seifert surface $\Sigma$ of $K$. Full details are given in Section \ref{Arithmetic}. 

\begin{definition} \label{WittSpan}
For a knot $K$ we define its {\em (rational) Witt span}, denoted $ws(K)$, as 
$$ws(K) = \min \{ \text{rank } B \, |\,  \partial [B] = \partial (W(K))\}.$$
In the above, $B$ ranges over all finite rank, non-degenerate, symmetric, bilinear forms over $\mathbb Q$, and $[B]$ denotes its equivalence class in $W(\mathbb Q)$. Said differently, the Witt span of a knot is the smallest rank of any form $B$ such that its equivalence class $[B]\in W(\mathbb Q)$ has the same image in $\oplus_{p \text{ prime}} W(\mathbb F_p)$ as $W(K)$ under the epimorphism $\partial$ from the sequence \eqref{SESForW(Q)}. 
\end{definition} 
Since the Witt class $W(K)$ is a concordance invariant of $K$, then so is $ws(K)$. With this understood, here is our main result. 
\begin{theorem} \label{main}
For any knot $K$, the Witt span $ws(K)$ is a lower bound for its concordance crosscap number, that is 
$$ ws(K) \le \gamma_c(K).$$
The range of the Witt span function on the set of all knots equals $\{0, 1,2,3\}$.
\end{theorem}
To understand the perhaps surprising bound  $ws(K)\le 3$ from the preceding theorem, we digress a bit first. For a field $\mathbb F$, a bilinear form $B:V\times V \to \mathbb F$ defined on a finite dimensional $\mathbb F$-vector space $V$ is called {\em isotropic} if there exists a nonzero vector $x\in V$ with $B(x,x)\ne 0$. If no such vector exists, we say that $B$ is {\em anisotropic}. The {\em universal invariant} $u(\mathbb F) \in \mathbb N$ of the field $\mathbb F$ is defined as the unique integer $n$ such that all non-degenerate symmetric bilinear forms of rank $n+1$ over $\mathbb F$ are isotropic, but such that there is at least one non-degenerate symmetric bilinear anisotropic form of rank $n$ over $\mathbb F$ (if no such $n$ exists, set $u(\mathbb F) = \infty$). As will become clear from the proof of Theorem \ref{main}, the bound $ws(K)\le 3$ is a direct consequence of the fact that $u(\mathbb Q_p)=4$ for every prime $p$, where $\mathbb Q_p$ is the field of the $p$-adic numbers. Sections \ref{Arithmetic} and \ref{SectionOnProofs} provide the relevant details. 
\vskip2mm

Of the three non-trivial values that $ws(K)$ can attain, one can completely characterize knots $K$ with $ws(K)=1$. One can additionally provide sufficient conditions under which $ws(K)$ equals 2 or 3. As we shall see in Theorem \ref{TheoremAlgorithmForComputingws}, calculating $ws(K)$ is completely algorithmic but can be laborious. The results from the next two theorems can be seen as providing significant computational shortcuts for some knots. 

To state the first theorem, let $\mathbb F$ be any field, and let $\langle a \rangle$ denote the 1-dimensional bilinear form over $\mathbb F$ with $\langle a \rangle (x,y) = axy$. The equivalence class of $\langle a \rangle$ in $W(\mathbb F)$ will be denoted by $[a]$. If $(V_1, B_1)$ and $(V_2, B_2)$ are two bilinear forms over $\mathbb F$, let $B=B_1\oplus B_2$ denote the bilinear form $B:(V_1\oplus V_2) \times (V_1\oplus V_2)\to \mathbb F$, defined by $B((v_1, v_2), (w_1, w_2)) = B_1(v_1,w_1) + B_2(v_2,w_2)$. We use $n\cdot \langle a \rangle$ to denote the $|n|$-fold sum of $\langle a\rangle$ with itself if $n\ge 0$, or $|n|$-fold sum of $\langle -a \rangle$ with itself if $n\le  0$ (the case of $n=0$ yields the zero form). The expression $n\cdot [a]$ is similarly defined. 
\begin{theorem} \label{Consequence1}
A knot $K$ has Witt span $ws(K)$ no greater than 1 if and only if 
$$W(K) = [d] \oplus (n\cdot [1])$$ 
with $d$ equal to either $\det K$ or $-\det K$,  and with $n=\sigma (K) - \text{Sign}(d)$. If $\det K$ is a square then $ws(K)=0$, otherwise $ws(K)=1$.  
\end{theorem}
For the next theorem recall that the map $\partial$ from \eqref{SESForW(Q)} is the direct sum of group homomorphisms $\partial _p:W(\mathbb Q) \to W(\mathbb F_p)$ with $p$ ranging over all primes. 
\begin{theorem} \label{TheoremCharacterizingWsTwoThree}
Let $K$ be a knot, let $p\equiv 1 \pmod 4$ be a prime, and let $\beta \in \mathbb F_p$ be any non-square. Assume that $\partial _q(W(K)) = 0$ for all primes $q\ne p$.
\begin{itemize}
\item[(i)] If $\partial _p(W(K)) = [ \beta]\in W(\mathbb F_p)$ then $ws(K)=2$.
\item[(ii)] If $\partial _p(W(K)) = [ 1 ] \oplus [ \beta] \in W(\mathbb F_p)$ then $ws(K)=3$.  
\end{itemize}
 
\end{theorem}
\subsection{Applications and Examples}  \label{SectonOnExamplesAndApplications}
Since $\gamma_4(K)$ is a lower bound for $\gamma_c(K)$ \eqref{Double Inequality}, the lower bound on $\gamma_c(K)$ from Theorem \ref{main} is useful for those knots $K$ with $\gamma_4(K)=1, 2$, and for the many more knots $K$ with unknown $\gamma_4(K)$. The examples below showcase such applications. Examples \ref{ExampleLowCrossingKnots} and \ref{ExampleProvingTheoremAboutPretzelKnots} are a consequence of Theorem \ref{TheoremCharacterizingWsTwoThree}, while Example \ref{Example940} follows from the more general Theorem \ref{TheoremAlgorithmForComputingws}.    
\begin{example}[Knots with up to 12 crossings]  \label{ExampleLowCrossingKnots}
We consider here the 2978 prime knots with crossing number up to 12. At the time of this writing $\gamma_4$ has only been computed for the 250 knots with crossing number up to 10, see \cite{Ghanbarian, JabukaKelly, KnotInfo}. Of the knots in this familily, those that satisfies criterion (ii) from Theorem \ref{TheoremCharacterizingWsTwoThree}, and which therefore have Witt span equal to 3 and $\gamma_c(K)\ge 3$, are the knots: 
$$
\begin{array}{c}
9_{49},\, \, 11n_{133}, \,\, 12a_{561},\, \, 12a_{664},\, \,  12a_{780}, \, \, 12a_{907},\,\, 12n_{276}.    
\end{array}$$
The choice of prime $p$ from Theorem \ref{TheoremCharacterizingWsTwoThree} is $p=5$ for all knots with the exception of $K=12a_{664}$ where $p=13$. To put these results into perspective, we list the known values of $\gamma_3$ and $\gamma_4$ for these knots from KnotInfo \cite{KnotInfo}:
$$
\begin{array}{|c|c|c|c|c|} \hline
\text{Knot } K & \gamma_4(K) & ws(K) & \gamma_c(K) & \gamma_3(K) \cr \hline  
9_{49} & 2  &3 & 3 & 3 \cr \hline
11n_{133} & \text{Unknown} & 3 & [3,4] & 4 \cr \hline
12a_{561} & \text{Unknown} &3 & [3,5] & 5 \cr \hline
12a_{664} & \text{Unknown} & 3 & [3,4] &4  \cr \hline
12a_{780} & \text{Unknown} &3 & [3,6] & 6 \cr \hline
12a_{907} & \text{Unknown} &3 & [3,5] &5  \cr \hline
12n_{276} & \text{Unknown} &3 & [3,4] & [2,4] \cr \hline
\end{array}
$$ 
Observe that $ws(12n_{276})=3$ implies that $\gamma_3(12n_{276})\ge 3$, improving on the current lower bound of 2 from \cite{KnotInfo}. The example of $K=9_{49}$ is the first example of a knot with $\gamma_4(K)=2$ and with $\gamma_4(K) < \gamma_c(K)$.
\end{example}
The preceding example was chosen for a comparison of our results to those of  \cite{Livingston}. Livingston points out in \cite{Livingston} that Theorem \ref{Livingston's Theorem} is very effective in obstructing $\gamma_c=1$ among low-crossing knots. Specifically, among the 801 prime knots with crossing number 11 or less,  74 knots are either slice or concordant to an alternating torus knot, and thus all have $\gamma_c$ equal to 1. Of the remaining 727 knots, Theorem \ref{Livingston's Theorem} can be used to obstruct all but 4 knots from having $\gamma_c$ equal to 1. These 4 knots are $9_{40}$, $11n_{45}$, $11n_{66}$ and $11n_{145}$. Given this, the next example serves as a comparison between Livingston's obstruction and our lower bound on $\gamma_c$ from Theorem \ref{main}. 
\begin{example}[Case of the knot $9_{40}$]  \label{Example940}
For the knot $K=9_{40}$ one obtains $ws(K) = 2$, leading to $\gamma_c(9_{40})\ge 2$. This shows that the bound $ws(K) \le \gamma_c(K)$ is independent from the obstruction to $\gamma_c(K)=1$ by Livingston's Theorem \ref{Livingston's Theorem}. 

It should be noted that Jabuka-Kelly \cite{JabukaKelly} showed that $\gamma_4(9_{40}) = 2$, a result that was not available at the time of publication of Livingston's work \cite{Livingston}, but a fact that can also be used to deduce $\gamma_c(9_{40})\ge 2$. 
\end{example}
We next apply Theorem \ref{TheoremCharacterizingWsTwoThree} to the family of pretzel knots $P(a_1, \dots, a_n)$ shown in Figure \ref{PretzelKnot}, with $a_1, \dots, a_n \in \mathbb Z-\{0\}$. In particular we show that there are infinitely many knots for which Theorem \ref{TheoremCharacterizingWsTwoThree} can be used to calculate $\gamma_c$. Note that $K=P(a_1, \dots, a_n)$ is a knot, as opposed to a link, precisely when either 
\begin{itemize}
\item[(i)] $K$ is an {\em odd pretzel knot}, that is $n$ is odd and each of $a_1, \dots, a_n$ is odd, or 
\item[(ii)] $K$ is an {\em even pretzel knot}, that is exactly one of $a_1, \dots, a_n$ is even, and $n$ may have either parity. 
\end{itemize} 
\begin{figure} 
\includegraphics[width=12cm]{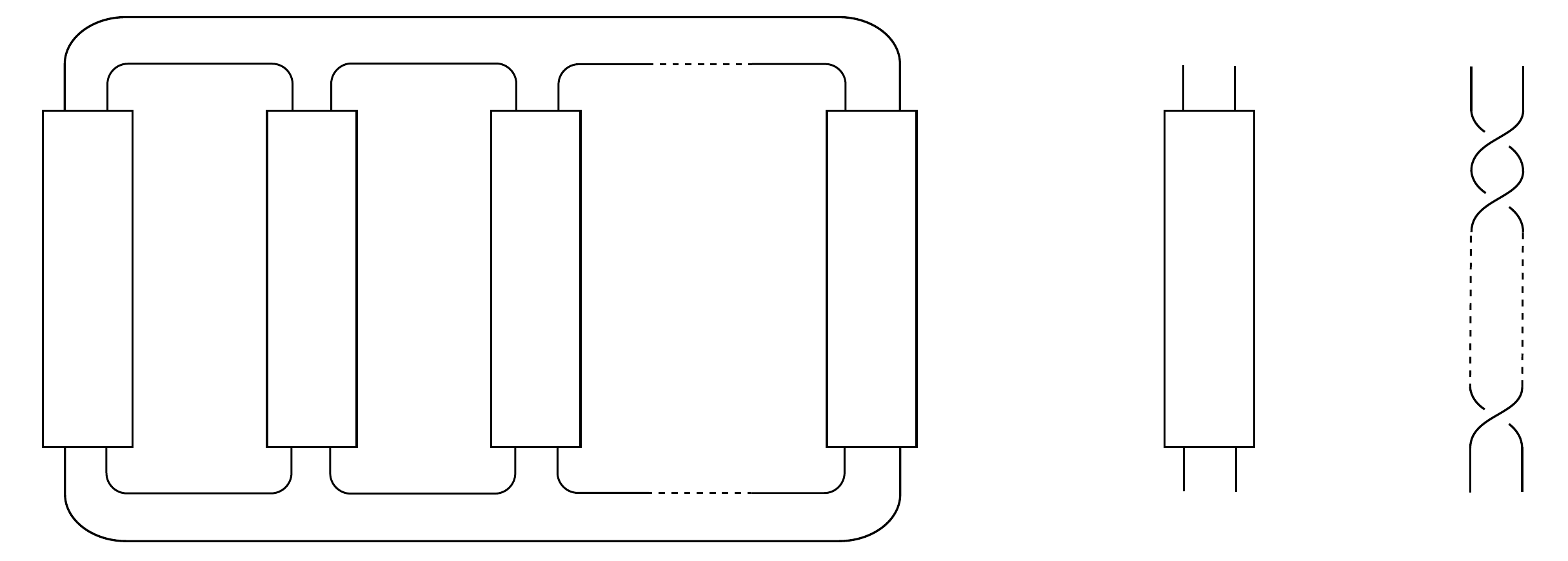}
\put(-327,60){$a_1$}
\put(-279,60){$a_2$}
\put(-230,60){$a_3$}
\put(-155,60){$a_n$}
\put(-80,60){$a$}
\put(-50,60){$=$}
\caption{The family of pretzel knots $P(a_1, \dots, a_n)$. Each 4-stranded box labeled with a nonzero integer $a$ represents a pair of strands with $a$ right-handed half-twists if $a>0$, and $-a$ left-handed half-twists is $a<0$.  } \label{PretzelKnot}
\end{figure}
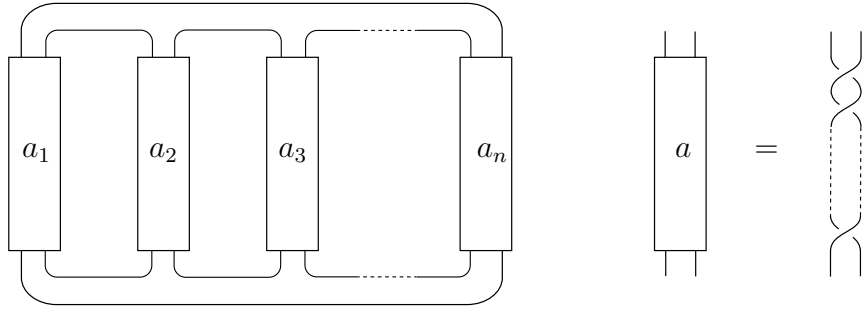
Concrete spanning surfaces for both odd and even pretzel knots as in Figure \ref{SpanningSurfacesForPretzelKnots}, make it clear that for $K=P(a_1,\dots, a_n)$ one has  
\begin{equation} \label{EquationUpperBoundsOnPretzelKnots}
ws(K) \le n-1 \qquad \text{ and } \qquad \gamma_c(K)\le\gamma_3(K)\le  \left\{
\begin{array}{cl}
n & \quad ; \quad K \text{ is odd,} \cr
n-1 &  \quad ; \quad K \text{ is even.}
\end{array}
\right. 
\end{equation}
%
\begin{figure} 
\includegraphics[width=12cm]{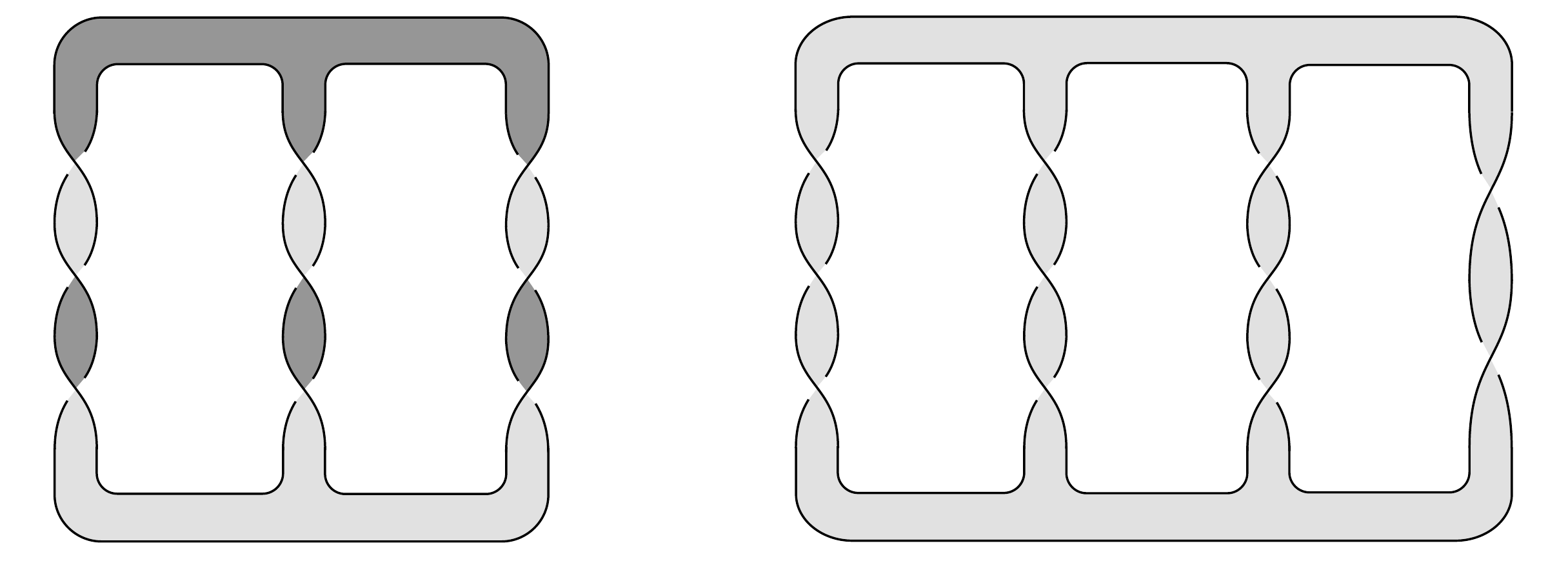}
\put(-283,-20){(a)}
\put(-96,-20){(b)}
\caption{The shaded spanning surface for the pretzel knot $K=P(a_1, \dots, a_n)$ is orientable of genus $(n-1)/2$ if $K$ is an odd pretzel knot (Subfigure (a) where $K=P(-3,-3,3)$), or nonorientable with first Betti number $n-1$ if $K$ is an even pretzel knot (Subfigure (b) where $K=P(-3,-3,3,2)$).   } \label{SpanningSurfacesForPretzelKnots}
\end{figure}
Indeed, it was shown by Ichihara and Mizushima \cite{IchiharaMizushima} that $\gamma_3(P(a_1,\dots, a_n))$ equals $n$ (respectively $n-1$) if $P(a_1, \dots, a_n)$ is an odd (respectively an even) pretzel knot. Of interest to us is the case of an even knot $K=P(a_1, \dots, a_n)$ with $ws(K) = n-1$, as for any such knot one obtains $\gamma_c(K) = n-1$ as well. We show that this happens for infinitely many even pretzel knots with $n=3, 4$.  
\begin{theorem} \label{TheoremAboutPretzelKnots}
Let $n$ equal to 3 or 4, then there exist infinitely many even $n$-stranded pretzel knots with Witt span and nonorientable concordance genus equal to $n-1$. 
\end{theorem}
The rational Witt classes for all pretzel knots have been computed in \cite{JabukaWitt}. Relying on these calculations, the next example proves the preceding theorem. 
\begin{example} \label{ExampleProvingTheoremAboutPretzelKnots}
Examples of families of even 3-stranded pretzel knots that satisfy the conclusions of the preceding theorem are given by $K_p= P(p,p,2p)$, with $p$ any prime congruent to $1 \pmod 5$ and congruent to $1 \pmod 4$ (e.g. any of the infinitely many primes in the arithmetic progression $1+20k$, $k\ge 1$).  
The rational Witt classes of these knots are easily computed to be (cf. \cite{JabukaWitt})
:%
$$W(K_p) = [2p]\oplus [10p]\oplus(2p\cdot [-1]) .$$
It follows that 
$$ 
\partial _q W((K_p)) =  \left\{ 
\begin{array}{cl}
[2]\in W(\mathbb F_{5}) & \quad ; \quad q=5, \cr
0 & \quad ; \quad   q\ne 5. 
\end{array}
\right.
$$
Since 2 is not a square modulo 5, Part (i) of Theorem \ref{TheoremCharacterizingWsTwoThree} shows that then $ws(P(p,p,2p))=2$ and thus also that $\gamma_c(P(p,p,2p)) = 2$.
\vskip3mm
Next consider the even 4-stranded pretzel knots $K_m = P(5,5,5^{2m+1}-2, 10)$ with $m\ge 0$. From \cite{JabukaWitt} one obtains 
$$W( K_m) = [10]\oplus [2]\oplus [5]\oplus ((18+5^{2m+1})\cdot [-1]).$$
It follows that $\partial_5(W(K_m)) = [1]\oplus [2] \in W(\mathbb F_5)$ and if $r\ne 5$ is any other prime, then $\partial_r(W(K_m)) = 0$. Part (ii) of Theorem \ref{TheoremCharacterizingWsTwoThree} implies that $ws(K_m) =3$ and thus $\gamma_c(K_m) = \gamma_3(K_m)= 3$ for all $m\ge 0$. This proves Theorem \ref{TheoremAboutPretzelKnots}. 
\end{example}
\subsection{Concluding remarks}
The considerations from Section \ref{SectonOnExamplesAndApplications} and the results of Theorem \ref{TheoremAboutPretzelKnots} prompt the following conjecture.  
\begin{conjecture}
For every $n\ge 3$ there exist infinitely many even $n$-stranded pretzel knots $K=P(a_1,\dots, a_n)$ with $\gamma_c(K) = n-1$. 
\end{conjecture}
At present there are no lower bounds on $\gamma_c$ effective for $n\ge 5$ that may facilitate a proof of this postulation. 
\vskip3mm
Recall the double inequality \eqref{Double Inequality}, by which $\gamma_4(K) \le \gamma_c(K) \le \gamma_3(K)$. Examples of knots $K$ for which $\gamma_3(K) - \gamma_4(K)$ can become arbitrarily large are known, see for instance \cite{JabukaVanCott}. Thus at least one of $\gamma_3(K) - \gamma_c(K)$ or $\gamma_c(K) - \gamma_4(K)$ also must grow without bound for such knots. It would be interesting to see if both quantities can be made unbounded simultaneously, prompting this question. 
\begin{question}
For each $n\in \mathbb N$ does there exist a knot $K_n$ with 
$$\gamma_3(K_n) - \gamma_c(K_n) \ge n \qquad \text{ and } \qquad \gamma_c(K_n) -\gamma_4(K_n)\ge n?$$
\end{question}
\subsection{Organization} In Section \ref{Arithmetic} we give a detailed review of the theory of symmetric bilinear forms over fields. The goal of Sections \ref{SectionValuationTheory} and \ref{SectionOnBilinearForms}  is a complete and algorithmic understanding of when a bilinear symmetric form over $\mathbb Q$ is (an)isotropic, a goal that is achieved by a combination of  Corollary \ref{HasseMinkovskiOverQ} and Theorems \ref{FormsOverFiniteFields} and \ref{TheoremCharacterizingSmallRankFormsOverQp}. Section \ref{SectionOnWittRings} gives background on Witt rings, while Section \ref{SectionOnProofs} provides the proofs for  Theorems \ref{main}, \ref{Consequence1} and \ref{TheoremCharacterizingWsTwoThree}. Section \ref{SectionOnProofs} also justifies the claim made in Example \ref{Example940}.  
\section{Arithmetic of Bilinear Forms over Fields} \label{Arithmetic}
This section provides background on the arithmetic of bilinear forms over fields.  The main goal is to understand properties of bilinear forms over $\mathbb Q$, an endeavor which requires an understanding of bilinear forms over $\mathbb Q_p$, the field of $p$-adic numbers, and quadratic forms over $\mathbb R$.  The discussion in this section follows those from \cite{Cassels, Gerstein, Hasse, Lam, MilnorHusemoller, OMeara}.

\subsection{Valuations, Completions and the Residue Class Field} \label{SectionValuationTheory}
\subsubsection{Valuations}
A {\em valuation $v$ on a field $\mathbb F$} is a function $v:\mathbb F\to \mathbb R$ subject to the following requirements, all which need to hold for all $a, b \in \mathbb F$:
\begin{itemize}
\item[(i)] $v(a) \ge 0$ and $v(a) = 0$ if and only if $a=0$. 
\item[(ii)] $v(ab) = v(a)\cdot v(b)$. 
\item[(iii)] $v(a+b) \le v(a) +  v(b)$.
\end{itemize}
A valuation $v$ is called {\em non-Archimedian} if instead of (iii) it satisfies the stronger condition of 
\begin{itemize}
\item[(iii)'] $v(a+b) \le \max\{v(a), v(b)\}, $
\end{itemize}
otherwise it is called an {\em Archimedean valuation}. The pair $(\mathbb F, v)$ is called a {\em valuated field}. For later use note that $v(1) = 1$, this is implied by properties (i) and (ii) above. A direct consequence of this is that $v(a^{-1}) = v(a)^{-1}$ for all nonzero $a\in \mathbb F$, a relation we shall tacitly rely on below.  

Two valuations $v_1$ and $v_2$ on $\mathbb F$ are called {\em equivalent} if there exists an $\eps>0$ such that $v_2(a) = (v_1(a))^\eps$ for all $a\in \mathbb F$.  This is an equivalence relation on the set of valuations on $\mathbb F$, and the equivalence class $[v]$ of a valuation is called a {\em  spot} or {\em prime} or {\em prime spot}  or {\em place} on $\mathbb F$. We shall refer to it as a spot in this exposition. The set of all spots on the field $\mathbb F$ is denoted by $\Omega$ or $\Omega (\mathbb F)$ for emphasis. We say that a valuation $v'$ represents the spot $[v]$ if $[v']=[v]$. 

\begin{example}[The Trivial Valuation]
Every field possesses valuations. One example is that of the {\em trivial valuation} $v_0$ defined as 
$$v_0(a) = \left\{
\begin{array}{cl}
0 & \quad ; \quad a=0,\cr
1 & \quad ; \quad a\ne 0. 
\end{array}
\right.$$
Note that the trivial valuation is non-Archimedean. The only representative of the {\em trivial spot $[v_0]$} is the trivial valuation $v_0$. 
\end{example}
\begin{example}[Valuations on $\mathbb Q$]  \label{ValuationsOnQ}
For each prime $p$, consider the function $v_p:\mathbb Q\to \mathbb R$, obtained as follows: Given a rational number $r\ne 0$, write is as $r = p^n \frac{a}{b}$ with $n, a, b\in \mathbb Z$, $b\ne 0$, and with $a$ and $b$ not divisible by $p$. Then $v_p$, called the {\em $p$-adic valuation on $\mathbb Q$}, is defined as
$$ v_p(r) = v_p\left( p^n \frac{a}{b} \right) = \frac{1}{p^n} \quad \quad \text{ and } \quad \quad v_p(0) = 0.$$
It is easy to see that  each $v_p$ is a non-Archimedean valuation on $\mathbb Q$. Additionally, let $v_\infty:\mathbb Q \to \mathbb R$ be the usual absolute value function on $\mathbb Q$. This too is a valuation on $\mathbb Q$, albeit an Archimedean one. These examples of valuations generate all spots on $\mathbb Q$, that is 
$$\Omega (\mathbb Q) = \{ [v_p]\,\,  |\, \, p \text{ is prime or } p=0 \text{ or } p=\infty\},$$
a fact known as ``Ostrowski's Theorem'', dating back to 1916 \cite{Ostrowski}.
\end{example}
\subsubsection{Completions}
A valuation $v$ gives $\mathbb F$ the structure of a metric space, with distance function $d_v(a,b) = v(a-b)$, $a,b \in \mathbb F$. In particular, $v$ endows $\mathbb F$ with a topology, and we say that $(\mathbb F, v)$ is a {\em complete valuated field} if $(\mathbb F, d_v)$ is a complete metric space, i.e. one in which every Cauchy sequence is convergent. A {\em completion} of a valuated field $(\mathbb F, v)$ is a complete valuated field $(\mathbb E, w)$ such that 
\begin{itemize}
\item[(i)] $\mathbb F$ is a subfield of $\mathbb E$ and $w|_\mathbb{F} = v$.
\item[(ii)] $\mathbb F$ is dense in $\mathbb E$ (in the topological sense). 
\end{itemize} 
\begin{theorem}[\cite{OMeara}, Pages 10--11]
Every valuated field $(\mathbb F, v)$ has a completion $(\mathbb E, w)$. This completion is unique in the sense that if $(\mathbb E, w)$ and $(\mathbb E',w')$ are two completions of $(\mathbb F, v)$, then there exists a field isomorphism $\varphi :\mathbb E \to \mathbb E'$ which is the identity on $\mathbb F$ and such that $w(x) = w'(\varphi(x))$ for all $x\in \mathbb E$.  
\end{theorem}
%
\begin{remark}
The completion of valuated field $(\mathbb F, v)$ only depends on the spot $[v]$, in the sense that if $v'$ is another valuation on $\mathbb F$ with $[v']=[v]$, and $(E,w)$ and $(E',w')$ are completions of $(F,v)$ and $(\mathbb F, v')$ respectively, then there exists an isomorphism $\varphi :E\to E'$, fixing $F$, and such that $[w] = [w'\circ \varphi]$. When convenient, we shall refer to the  completion of $\mathbb F$ at one of its spots.     
\end{remark}
There are different ways of constructing the completion $(\mathbb E,w)$ of $(F,v)$. For example, one can let $\mathbb E$ be the set of equivalence classes $[x]$ of Cauchy sequences $x=(x_1, x_2, \dots )$ in $(\mathbb F,v)$, with $(x_1, x_2, \dots)$ equivalent to $(y_1,y_2, \dots)$ if $\lim v(x_n)-v(y_n) = 0$. The field $\mathbb F$ embeds into $\mathbb E$ by the obvious map $a\mapsto (a,a,a,\dots)$. The valuation $v$ extends to a valuation $w$ on $\mathbb E$ via $w((x_1, x_2, \dots)) = \lim v(x_n)$. For us however, the more practical approach to describing elements from $\mathbb E$ shall be that using the canonical power series representation, applicable when $v$ is a discrete valuation, see Theorem \ref{TheoremGersteinOnCanonicalPowerSeriesRepresentation}. 
\begin{example} \label{ExampleFirstIntroductionOfThePAdicNumbers}
Let $p$ be either a prime or else $p=0, \infty$. Of the valuated fields $(\mathbb Q, v_p)$ only $(\mathbb Q, v_0)$ is complete, with $d_{v_0}$ being the discrete metric on $\mathbb Q$. The completion of $(\mathbb Q, v_\infty)$ is $(\mathbb R, |\cdot|)$. The unique completion $(\mathbb Q_p, w_p)$ of $(\mathbb Q, v_p)$ for $p\ne 0,\infty$  is called the {\em  field of the $p$-adic numbers}. 
\end{example}
%
\subsubsection{The  Residue Class Field and Discrete Valuations}
For a non-Archimedean valuated field $(\mathbb F, v)$ we define 
\begin{align} \label{RingOfIntegers}
A & =A(v) = \{ a\in \mathbb F\, |\,  v(a) \le 1\} = \text{ The {\em valuation ring} or  {\em ring of integers} of $\mathbb F$ at $v$}. \cr
U & = U(v) = \{a\in \mathbb F\, |\, v(a) = 1\} = \text{ The {\em (multiplicative) group of units of} }A. \cr
\mathfrak p & = \mathfrak p(v) = \{ a\in \mathbb F\, |\,  v(a) <1\} = \text{ The unique maximal ideal of $A$}. 
\end{align}
Note that $A$, $U$ and $\mathfrak p$ are independent of the choice of $v$ from $[v]$, these objects only depend on the spot determined by $v$. 

It is easy to show that $A$ is a subring of $\mathbb F$. If $a\in A$ is a unit with inverse $b\in A$, then $1=v(1) = v(ab) = v(a)v(b)$, which together with $v(a), v(b) \ge 1$ (by virtue of being elements of $A$) implies that $v(a) = 1 = v(b)$. Conversely, if $a\in A$ and $v(a) = 1$, let $b\in \mathbb F$ be the inverse of $a$. Then again $1=v(1) =v(ab) = v(a) v(b) = v(b)$ showing that $b\in A$. This shows that $U$ is the group of units of $A$. 

Since $\mathfrak p = A-U$, it is clear that $\mathfrak p$ is a maximal ideal in $A$. It is the unique maximal ideal in $A$ for if $\mathfrak a$ is any other ideal of $A$ not contained in $\mathfrak p$, it would have to contain an element $a\in A-\mathfrak p = U$. Being a unit $a$ generates all of $A$ showing that $\mathfrak a = A$. 

The field 
$$\overline{\mathbb F}: = A/\mathfrak p,$$
is called the {\em residue class field} of $\mathbb F$ at $v$. The quotient map $A\to \overline{\mathbb F}$ shall be written as 
$$A\ni a\mapsto \bar a = a+\mathfrak p \in \overline{\mathbb F}.$$
\vskip2mm
Every valuation $v$ on $\mathbb F$ induces a group homomorphism $v:(\mathbb F^\ast, \cdot) \to (\mathbb R^+, \cdot)$ (with $\mathbb F^\ast = \mathbb F-\{0\}$, as usual). The image of this homomorphism is either dense in $\mathbb R^+$ or a subgroup of $(\mathbb R^+,\cdot)$ isomorphic to $\mathbb Z$ (cf. \cite{OMeara}, Page 37). In the latter case we call $v$ a {\em discrete valuation on $\mathbb F$}, and we refer to $(\mathbb F, v)$ as a {\em discretely valuated field}. If additionally $(\mathbb F, d_v)$ is complete, we call $(\mathbb F, v)$ a {\em complete discretely valuated field}. If $v$ is discrete, it is automatically non-Archimedean. 

Let $v$ be a discrete valuation on $\mathbb F$ and let $\pi\in \mathfrak p$ be an element maximizing the value of $v|_{\mathfrak p}$ (a definition that would be ill posed if $v$ were not discrete). Any such element $\pi$ is called a {\em uniformizer of $A$}. Uniformizers are unique up to multiplication by units of $A$. Indeed, if $\pi'$ is another uniformizer of $A$, then $v(\pi) = v(\pi')$ showing that $v(\pi'\cdot \pi^{-1}) = 1$. It follows that $\pi'=u\pi$ for some unit $u\in U$. 

If $a\in \mathfrak p$ is any element, then there exists some $n\in \mathbb N$ such that $v(\pi^n) = v(a)$, showing that $a=u\pi^n$ for some unit $u\in U$. This show that $\mathfrak p = (\pi)$, i.e. $\mathfrak p$ is a principal ideal generated by any uniformizer $\pi$. The integer $n$, called the {\em $\mathfrak p$-adic order of $a$} and denoted by $ord_p(a)$, can be explicitly calculated using any uniformizer $\pi$ of $\mathfrak p$, as $ord_p(a) = n=\ln v(a) / \ln v(\pi)$. 
\begin{theorem} [\cite{Gerstein}, Page 64] \label{TheoremGersteinOnCanonicalPowerSeriesRepresentation}
Let $(\mathbb F, v)$ be a discretely valuated field, let $\pi$ be a uniformizer of $A$, and let $C$ be a representative set of the residue class field $\overline{\mathbb F}$ (the selection of which may require the use of the Axiom of Choice). Then 
\begin{itemize}
\item[(a)] Every $a\in \mathbb F^\ast$ has a unique expression $a=\sum _{i=\nu} ^\infty c_i \pi^i$ with $c_i \in C$ and $\nu = ord_p(a)$. 
\item[(b)] If $(\mathbb F, v)$ is a complete discretely valuated field, then every such power series converges in $\mathbb F$. 
\end{itemize}
The power series from (a) is called the ``canonical power series representation of $a$". 
\end{theorem}
\begin{example} \label{ExampleCanonicalPowerSeriesRepresentationForThePAdicNumbers}
For each prime $p$, $(\mathbb Q, v_p)$ is a discretely valuated field and $\pi = p$ is a uniformizer for $\mathfrak p$. Let as in Example \ref{ExampleFirstIntroductionOfThePAdicNumbers}, $(\mathbb Q_p, w_p)$ be its unique completion. Then the residue class field of $\mathbb Q_p$ is $\mathbb F_p$, allowing us to pick $C=\{0,\dots, p-1\}$. Then every element of $a \in \mathbb Q_p^\ast$ has a unique power series representation
$$ a=\sum _{i={ord_p(a)}}^\infty c_i \,p^i$$
with $c_i\in \{0,\dots, p-1\}$, and each such power series converges in $\mathbb Q_p$. 
\end{example}
\subsection{Bilinear Forms}  \label{SectionOnBilinearForms}
This section introduces the theory of symmetric bilinear forms over fields. The theory is sensitive to whether the field is nondyadic (of characteristic not equal to 2) or dyadic (characteristic equal to 2). The latter, being the more difficult case, is treated separately in Section \ref{SectionForAnisotropyOverQ2}. The two main goals of this section are to understand when a bilinear form over $\mathbb Q$ is isotropic (Corollary \ref{HasseMinkovskiOverQ}), and to fully understand the Witt ring $W(\mathbb Q)$ in terms of the Witt rings $W(\mathbb F_p)$ (Theorem \ref{TheoremSplitExactSequenceForWittGroupOfRationals}).      
\subsubsection{Bilinear Forms over Fields} \label{SectionQuadraticFormsOverFields}
Let $\mathbb F$ be a field. A {\em bilinear form $B$ over $\mathbb F$} is a pair  $(V,B)$ consisting of a finite dimensional $\mathbb F$-vector space $V$ and a bilinear function $B:V\times V \to \mathbb F$ such that $B(x,y) = B(y,x)$ for all $x, y \in V$. The {\em associated quadratic form $q_B$} is the function $q_B:V\to \mathbb F$ defined by $q_B(x) = B(x,x)$. When char$(\mathbb F)\ne 2$, $B$ can be recovered from $q_B$ as 
$$B(x,y) = \textstyle \frac{1}{2} (q_B(x+y) - q_B(x) -q_B(y)).$$
Classically, quadratic forms over a field $\mathbb F$ were defined as homogeneous, degree 2 polynomials in $n\ge 1$ variables and with coefficients in $\mathbb F$. When char$(\mathbb F)\ne 2$, any such polynomial $f(x_1,\dots, x_n)$ can be written as 
$$f(x_1, \dots, x_n) = \sum _{i,j = 1}^n a_{i,j} \, x_i\,  x_j$$
with $a_{i,j}\in \mathbb F$ and with $a_{i,j} = a_{j,i}$ (if this condition is not met at first, simply replace $a_{i,j}$ with $\frac{1}{2}(a_{i,j}+a_{j,i})$). Then the matrix $A=[a_{i,j}]$ is a symmetric $n\times n$ matrix, and can be viewed as defining the symmetric bilinear form $B$ on $V=\mathbb F^n$, given by $B(x,y) = \langle Ax, y\rangle$ (where $Ax$ refers to matrix multiplication and $\langle \cdot, \cdot \rangle$ the standard scalar product on $\mathbb F^n$). It is easy to check that $q_B(x) = f(x_1, \dots , x_n)$, underscoring that the one-to-one correspondence between symmetric bilinear forms and quadratic forms. This correspondence breaks down in the case of dyadic fields.
\begin{remark}
All bilinear forms appearing below will always be assumed to by symmetric. We shall thus refer to them simply as {\em bilinear forms} or just {\em forms} for simplicity. We will typically write $B$ to refer to the bilinear form $(V,B)$. 
\end{remark}
The bilinear form $B$ is called {\em non-degenerate} or {\em regular} if whenever $B(x,y) = 0$ for all $y\in V$, it follows that $x=0$. A form that is not regular is called {\em degenerate} or {\em singular}. Going forward we shall exclusively be interested in non-degenerate bilinear forms.  For a non-degenerate form $(V,B)$, we define its rank as rank$(B) = \dim_{\mathbb F}V$. A bilinear form $(V,B)$ may be represented by a symmetric matrix $[b_{i,j}]$ with respect to a choice of a basis of $\{e_1, \dots, e_n\}$ of $V$, by setting $b_{i,j} = B(e_i,e_j)$. Then $B$ is non-degenerate if and only if $\det [b_{i,j}]\ne 0 $. 

The {\em discriminant of B}, denoted $dB$, if the coset $\det [b_{i,j}] \cdot \mathbb (F^\ast)^2$ in $\mathbb F^\ast/(\mathbb F^\ast)^2$. It is well defined since if we pick a different basis for $V$ and use it to represent $B$ by its associated matrix $[b'_{i,j}]$, then $[b'_{i,j}] = A\cdot [b_{i,j}]\cdot A^\tau$ for some invertible matrix $A$. Clearly then $\det [b'_{i,j}] = \det [b_{i,j}]\cdot (\det A)^2$. By way of notational shortcut, we shall typically write $dB=a$ rather than $dB = a\cdot (\mathbb F^\ast)^2$. 

A bilinear form $B$ is called {\em isotropic} if there exists a nonzero vector $x\in V$ with $B(x,x)=0$, in which case $x$ itself is called an {\em isotropic vector} for $B$. If $B$ is not isotropic, we call it {\em anisotropic}. A form $B$ is called {\em alternating} if $B(x,x)=0$ for all $x\in V$. Clearly alternating forms are isotropic. 
\begin{example} \label{ExampleHyperbolicPlane}
For any field $\mathbb F$, the {\em hyperbolic form} $\mathbb H = \mathbb H(\mathbb F)$ is the symmetric bilinear rank 2 form defined as 
$$\mathbb H((x_1, y_1), (x_2,y_2)) = x_1y_2+ x_2y_1.$$ 
$\mathbb H$ is alternating when char$(\mathbb F)=2$.
\end{example} 

Given a bilinear form $B:V\times V \to \mathbb F$, a subspace $W\subset V$ is called {\em totally isotropic} (with respect to $B$) if $B(x,y) = 0$  for all $x, y \in W$. A bilinear form $(V,B)$ is called {\em metabolic} if $dim_\mathbb F V = 2k$ for some $k\ge 0$, and $V$ possesses a totally isotropic subspace $W$ of dimension $k$. It is not hard to see that for a non-degenerate bilinear form $B$ of rank $2k$, a totally isotropic subspace of $V$ can have dimension at most $k$. Thus, a metabolic form $(V,B)$ is one admitting a maximal size totally isotropic subspace. As we shall see later, these two types of ``extremal" bilinear forms, namely anisotropic ones (possessing no nontrivial totally isotropic subspaces) and metabolic ones (possessing maximal size totally isotropic subspaces) , generate all bilinear forms by means of direct sums. 

Two bilinear forms $(V_1,B_1)$ and $(V_2,B_2)$ are called {\em isomorphic}, and we write $B_1 \cong B_2$, if there exists a vector space isomorphism $\varphi:V_1\to V_2$ such that $B_2(\varphi (x), \varphi(y)) = B_1(x,y)$ for all $x, y \in V_1$. For the most part, we shall not distinguish between isomorphic forms. 

A bilinear form $B$ is said to {\em represent $\alpha \in \mathbb F$} if there exists a vector $x\in V$ such that $B(x,x) = \alpha$. A bilinear form is called {\em universal} if it represents every $\alpha \in \mathbb F$. 

For $a\in \mathbb F$ we let the bilinear form $(\mathbb F, \langle a\rangle)$ (which we denote $\langle a \rangle$ for short) be defined by with $\langle a \rangle (x,y) = axy$. Note that  the  forms $(\mathbb F, \langle \lambda ^2\cdot a\rangle)$ and $(\mathbb F, \langle a \rangle)$ are isomorphic for any choice of $\lambda \in \mathbb F^\ast$, the isomorphism being $\varphi_\lambda :\mathbb F\to \mathbb F$ given by $\varphi_\lambda(x) = \lambda \cdot x$.  

The {\em (direct or orthogonal) sum} of a pair of bilinear forms $(V_1, B_1)$ and $(V_2, B_2)$ is the bilinear form $(V_1\oplus V_2, B_1\oplus B_2)$ with $B_1\oplus B_2((x_1,y_1),( x_2,y_2)) = B_1(x_1,y_1)+B_2(x_2,y_2)$. It is easy to check that $B_1\oplus B_2$ is non-degenerate if and only if $B_1$ and $B_2$ are non-degenerate. Similarly, the {\em product} of $(V_1, B_1)$ and $(V_2, B_2)$ is the bilinear form $(V_1\otimes_{\mathbb F} V_2, B_1\otimes B_2)$ with $B_1\otimes B_2(\sum_i x_i\otimes y_i, \sum_j x'_j\otimes y'_j) = \sum _{i,j} B_1(x_i,x'_j)\cdot B_2(y_i,y'_j)$.

Of great importance in the study of bilinear forms is the fact that for any non-degenerate bilinear form $B$ is a direct sum of 1-dimensional forms. 
\begin{theorem}[\cite{Lam}, Page 7 and \cite{Szymiczek}, Page 69] \label{TheoremDiagonalizabilityOfQuadraticForms}
Let $\mathbb F$ be a field and let $B:V\times V\to \mathbb F$ be a non-degenerate bilinear form of rank $n$. If char$(\mathbb F)\ne 2$ there exist elements $a_1, \dots, a_n\in \mathbb F^\ast$ such that 
\begin{equation} \label{EquationDiagonalizabilityOfForms}
B \cong  \langle a_1 \rangle \oplus \dots \oplus \langle a_n\rangle. 
\end{equation}
The same is true if char$(\mathbb F)=2$, provided $B$ is not an alternating form.
\end{theorem}
This theorem is false for alternating forms over dyadic fields. For instance, the hyperbolic form $B=\mathbb H(\mathbb F_2)$ from Example \ref{ExampleHyperbolicPlane} cannot be diagonalized. 
\begin{corollary} \label{CorollaryOnDiagonalizabilityOfAnisotropicForms}
Anisotropic symmetric bilinear forms $B$ over any field satisfy \eqref{EquationDiagonalizabilityOfForms}. 
\end{corollary}
The next proposition gives a characterization of the representability of  $\alpha \in \mathbb F$ by a bilinear form $B$, a proof of which can be found in \cite{Lam}, Page 11. 
\begin{proposition}
A non-degenerate bilinear form $B:V\times V\to \mathbb F$ represents $\alpha \in \mathbb F$ if and only if $B\oplus \langle -\alpha\rangle$ is isotropic. 
\end{proposition}
\subsubsection{Bilinear Forms over Local Fields} \label{SectionQuadraticFormsOverLocalFields}
Let $(\mathbb F, v)$ be a complete discretely valuated field, let  $A$, $U$ and $\mathfrak p$ be as in \eqref{RingOfIntegers} and let $\pi$ be a uniformizer of $A$. Let $B$ be an $n$-dimensional non-degenerate bilinear form over $\mathbb F$ and assume that either char$(F)\ne 2$ or that $B$ is not alternating if char$(\mathbb F)=2$. Then there exist units $u_1, \dots, u_r, u_{r+1},\dots , u_n \in U$ such that 
\begin{equation} \label{QuadraticFormsOverCompleteDiscretelyValuatedFields}
B\cong \langle u_1\rangle \oplus \dots \oplus \langle u_r\rangle \oplus \langle u_{r+1}\pi\rangle \oplus \dots \oplus \langle u_n\pi\rangle.
\end{equation}
This follows from Theorem \ref{TheoremDiagonalizabilityOfQuadraticForms} and the facts:
$$
\begin{array}{rlcrl}
\mathbb F^* & = A^\ast \cup (A^\ast)^{-1}  & \qquad  \qquad \qquad & A& =U\cup \mathfrak p \cr
\langle a \rangle & \cong \langle a^{-1}\rangle     & & \mathfrak p & = (\pi) \cr
\langle u\pi^m\rangle&  \cong \left\langle u \pi^{m \,(\text{mod } 2)}\right\rangle & & &
\end{array}
$$
Thus in our study of non-degenerate quadratic forms over complete discretely valuated non-dyadic fields, it suffices to restrict attention to those of the form \eqref{QuadraticFormsOverCompleteDiscretelyValuatedFields}. The same is true over dyadic fields for non-alternating forms. 
\begin{proposition}[\cite{Lam}, Page 148] \label{PropositionAnisotropicForms}
Let $(\mathbb F, v)$ be a complete, discretely valuated field with char$(\mathbb F)\ne 2$, let $u_1,\dots, u_r, u_{r+1},\dots, u_n\in U$ be units and let $\pi \in \mathfrak p$ be a uniformizer.  
\begin{itemize}
\item[(a)] The bilinear form $\langle u_1\rangle \oplus \dots \oplus \langle u_r\rangle$ is anisotropic over $\mathbb F$ if and only if the bilinear form $\langle \bar u_1\rangle \oplus \dots \oplus \langle \bar u_r\rangle$ is anisotropic over $\overline{\mathbb F}$.  
\item[(b)] The bilinear form $\langle u_1\rangle \oplus \dots \oplus \langle u_r\rangle \oplus \langle \pi u_{r+1}\rangle \oplus \dots \oplus \langle \pi u_n\rangle$ is anisotropic over $\mathbb F$ if and only if the bilinear forms $\langle \bar u_1\rangle \oplus \dots \oplus \langle \bar u_r\rangle$ and $\langle \bar u_{r+1}\rangle \oplus \dots \oplus \langle \bar u_n\rangle$ are both anisotropic over $\overline{\mathbb F}$.  
\end{itemize}
\end{proposition}
If $\mathbb F$ is additionally a {\em local field}, that is a complete discretely valued field $(\mathbb F, v)$ whose residue class field $\overline{\mathbb F}$ is finite, more can be said about the isotropy property of bilinear forms. Before proceeding, we note that the most prominent examples of local fields, and the only ones relevant to our discussion, are the $p$-adic fields $(\mathbb Q_p,w_p)$, with $p$ a prime, in which case $\overline{\mathbb Q}_p =\mathbb F_p$.

\begin{theorem}[\cite{Lam}, Page 152--158] \label{IsotropicFormsOfDimension5}
Let $\mathbb F$ be a local field (of any characteristic) and let $B$ be an $n$-dimensional non-degenerate bilinear form over $\mathbb F$. 
\begin{itemize}
\item[(a)] If $n \ge 5$ then $B$ is isotropic. 
\item[(b)] There exists a unique, up to isomorphism, anisotropic bilinear form $B$ over $\mathbb F$, namely
$$B= \langle 1\rangle \oplus \langle -u \rangle \oplus \langle -\pi \rangle \oplus \langle u \pi\rangle.$$
Here $u\in U$ and $\pi$ is a uniformizer of $A$. 
\end{itemize}
\end{theorem} 

At least for the case of local fields $\mathbb F$ with char$(\mathbb F)\ne 2$, Proposition \ref{PropositionAnisotropicForms}, with the help of Theorem \ref{IsotropicFormsOfDimension5}, reduces the question of whether a non-degenerate bilinear form $B$ over $\mathbb F$ is anisotropic, to the question whether certain induced bilinear forms are anisotropic over the residue class field $\overline{\mathbb F}$. When $\mathbb F = \mathbb Q_p$ then $\overline{\mathbb F} = \mathbb F_p$, putting the onus on characterizing the anisotropy of bilinear forms over $\mathbb F_p$.  Towards such a characterization, let $p$ be an odd prime and let $\beta \in \mathbb F_p$ be a non-square. Note that $\mathbb F^\ast_p = (\mathbb F^\ast_p)^2\sqcup \beta \cdot  (\mathbb F^\ast_p)^2$. In particular, if $B$ is a  non-degenerate bilinear form over $\mathbb F_p$ with either $p\ne 2$ or with $B$ not alternating, then there exist integers $m, n \ge 0$ such that $B\cong (n\cdot \langle 1 \rangle) \oplus (m\cdot \langle \beta\rangle)$. 
\begin{theorem} \label{FormsOverFiniteFields}
Let $p$ be an odd prime  and let $\beta \in \mathbb F_p$ be a fixed non-square. 
\begin{itemize}
\item[(a)] Every non-degenerate bilinear form of rank $n\ge 1$ over $\mathbb F_p$ is isomorphic to one of the two mutually non-isomorphic forms: 
$$\underbrace{\langle 1 \rangle  \oplus \dots \oplus \langle 1 \rangle}_{n \text{ summands}} \qquad \text{ or } \qquad  \underbrace{\langle 1 \rangle  \oplus \dots \oplus \langle 1 \rangle}_{n-1  \text{ summands}} \oplus \langle \beta\rangle. $$ 
\item[(b)] Every non-degenerate bilinear form over $\mathbb F_p$ of rank at least 3, is isotropic.
\item[(c)] Both  $\langle 1\rangle$ and $\langle \beta \rangle$  are anisotropic. 
\item[(d)] Of the two rank 2 forms $\langle 1 \rangle \oplus \langle 1 \rangle$ and $\langle 1 \rangle \oplus \langle \beta \rangle$, the form 
\begin{itemize}
\item[(i)] $\langle 1\rangle\oplus \langle 1\rangle$ is isotropic and  $\langle 1 \rangle \oplus \langle \beta \rangle $ is anisotropic if $p\equiv 1 \pmod 4$.
\item[(ii)]  $\langle 1\rangle\oplus \langle 1\rangle$ is anisotropic and  $\langle 1 \rangle \oplus \langle \beta \rangle $ is isotropic if $p\equiv 3 \pmod 4$.
\end{itemize}
\end{itemize} 
Note that if $p\equiv 3 \pmod 4$, $\beta$ can be taken to equal $-1$. 
\end{theorem}
Parts (a) and (b) of this theorem can be found in \cite{OMeara}, pages 157--158, part (c) is trivial. For part (d), is suffices to examine the two rank 2 forms $\langle 1 \rangle \oplus \langle 1 \rangle$ and  $\langle 1 \rangle \oplus \langle \beta \rangle$. If $p\equiv 1\pmod 4$ then $-1\in (\mathbb F_p^\ast)^2$ so that $\langle 1\rangle \oplus \langle 1\rangle \cong \langle 1\rangle \oplus \langle -1\rangle$ and thus isotropic. If the form $\langle 1\rangle \oplus \langle \beta\rangle$ were isotropic, we could find an isotropic vector $x=(a,b)\in \mathbb F^2$, leading to $a^2+\beta b^2\equiv 0 \pmod p$. This shows that both $a$ and $b$ are non-zero, and thus that $\beta = -a^2b^{-2}$. Since $-1$ is a square in $\mathbb F_p$, this last equation equates $\beta $ with a square in $\mathbb F_p$, a contradiction. Thus $\langle 1 \rangle \oplus \langle \beta \rangle$ is anisotropic. The case of $p\equiv 3 \pmod 4$ follows along similar lines (the main difference being that $-1$ is a non-square in this case).  
\subsubsection{Anisotropy over $\mathbb Q_2$} \label{SectionForAnisotropyOverQ2}
Unfortunately Proposition \ref{PropositionAnisotropicForms} fails in the case of characteristic 2, leaving us without a test of anisotropy of non-degenerate bilinear forms over $\mathbb Q_2$. However, by Theorem \ref{IsotropicFormsOfDimension5} we know that any form of rank $n\ge 5$ over $\mathbb Q_2$ is isotropic, and that there is a unique non-degenerate anisotropic form or rank 4. It therefore only remains to understand  when a form $B$ over $\mathbb Q_2$ of rank 1, 2 or 3 is anisotropic. The goal of this section is to provide such criteria. Our outline follows that of \cite{Cassels, Gerstein}. 

Let $p$ be a prime and consider the field $\mathbb Q_p$ of the $p$-adic numbers. We shall rely on the canonical power series representation of elements $a\in \mathbb Q_p^\ast$ from Example \ref{ExampleCanonicalPowerSeriesRepresentationForThePAdicNumbers}, and write $a=\sum _{i=\nu}^\infty a_i \, p^i$ with $a_i\in \{0,\dots, p-1\}$, $\nu = \nu(a) = ord_p(a)$ and $a_\nu \ne 0$. Recall that a unit $u\in A$  has $\nu(u)=0$. Accordingly we can write $a=p^\nu u$ for some unit $u$, whose canonical power series representation is $u=\sum _{i=\nu}^\infty a_i \, p^{i-\nu}$. 

For a pair $a,b \in \mathbb Q_p^\ast$ we define its {\em Hilbert symbol} $(a,b)_p$ as 
$$(a,b)_p = \left\{
\begin{array}{rl}
1 & \quad ; \quad ax^2+by^2=1 \text{ is solvable in }\mathbb Q_p, \cr
-1 & \quad ; \quad ax^2+by^2=1 \text{ is not solvable in }\mathbb Q_p.
\end{array}
\right.$$  
For the case of $p=2$, the following makes the evaluation of the Hilbert symbol completely explicit (cf. \cite{Gerstein}, Page 84): Write $a=2^\nu \sum _{i=0}^\infty a_i \,2^i$  and $b=2^\mu \sum _{j=0}^\infty b_j\, 2^j$, then 
\begin{equation} \label{EquationTheHilbertSymbolFor2Adics}
(a,b)_2 = (-1)^{a_1b_1 + \nu\,  \omega (b) + \mu \, \omega (a)},
\end{equation}
with $\omega(a) = a_1+a_2 \pmod 2$ and $\omega(b) = b_1+b_2 \pmod 2$. Similar formulas exist for odd primes $p$ as well (\cite{Gerstein}, Page 83). 
\begin{example}
For $a=13$ and $b=24$, calculate $(a,b)_2$. Start by finding the canonical power series representation of $a$ and $b$:
$$a= 13 = 2^0( 1+0\cdot 2 + 1\cdot 2^2 + 1\cdot 2^3) \qquad \text{ and } \qquad b=24=2^3(1+1\cdot 2).$$
Accordingly $\nu=0$, $\mu = 3$, $a_1=0$, $a_2=1$, $b_1=1$ and $b_2=0$. Using these in \eqref{EquationTheHilbertSymbolFor2Adics} gives $(13, 24)_2 = (-1)^{0\cdot 1 + 0\cdot 1 + 3 \cdot 1} = -1$. Thus the equation $13x^2+24y^2=1$ is unsolvable in $\mathbb Q_2$. 
\end{example}
\begin{example}
Calcuate $(-1,-1)_2$. Note that
$$ -1 = \frac{1}{1-2} = \sum _{i=0}^\infty 1 \cdot2^i.$$
We can read off from this power series that $\nu=\mu = 0$ and $a_1=b_1=a_2=b_2=1$, leading to $(-1,-1)_2 = -1$.
\end{example}
Given a non-degenerate symmetric bilinear form $B=\langle a_1\rangle \oplus \dots \oplus \langle a_n\rangle$ over $\mathbb Q_p$, we define it {\em Hasse symbol} $S_pB$ as 
$$S_pB = \prod _{i<j} (a_i, a_j)_p.$$
The Hasse symbol of a form $B$ is independent of its chosen diagonalization (\cite{Gerstein}, Page 86), and can be explicitly calculated for $p=2$ using equation \eqref{EquationTheHilbertSymbolFor2Adics}. 
\begin{theorem} \label{TheoremCharacterizingSmallRankFormsOverQp}
Let $B$ be a non-degenerate form over $\mathbb Q_p$ of rank $r$. 
\begin{itemize}
\item[(i)] If $r=2$, then  $B=\langle a \rangle \oplus \langle b \rangle $, $a, b \in \mathbb Q_p^\ast$ is isotropic if and only if $-ab \in (\mathbb Q_2^\ast)^2$. 
\item[(ii)] If $r=3$, $B$ is isotropic if and only if $S_pB = (-1, -dB)_p$.
\item[(iii)] If $r=4$, $B$ is anisotropic if and only if $dB=1$ and $S_pB= -(-1,-1)_p$. 
\end{itemize}
\end{theorem}
Determining whether an element of $\mathbb Q_p^\ast$ is a square, something required in Part (i) of the preceding theorem, is thankfully easy. We state the result for $p=2$ only, though similar results exist for odd primes also (\cite{Gerstein}, Page 76). Let $a=2^\nu \sum _{i=0}^\infty a_i \, 2^i \in \mathbb Q_2^\ast$, then 
$$a \in (\mathbb Q_2^\ast)^2 \text{ if and only if $\nu$ is even and } a_1=a_2=0.$$ 
For example $-1$ is not a square in $\mathbb Q_2$, but 17 is. 

\begin{corollary} \label{CorollaryAboutAnisotropyOverQ2OfCertainRank4Forms}
Let $p$ be a prime, $a, b \in \mathbb Q_p^\ast$ and let $B_{\pm 1}$ be the two forms $B_{\pm 1}=\langle a \rangle \oplus \langle b \rangle \oplus\langle \pm 1 \rangle \oplus\langle \pm 1 \rangle$. 
\begin{itemize}
\item[(i)] If $p$ is odd, then $B_1$ is anisotropic if and only if $B_{-1}$ is anisotropic. 
\item[(ii)] If $p=2$ then $B_1$ is anisotropic if and only if $B_{-1}$ is isotropic.  
\end{itemize}
\end{corollary}
\begin{proof}
By part (iii) from Theorem \ref{TheoremCharacterizingSmallRankFormsOverQp}, it suffices to compare $S_pB_1$ to $S_pB_{-1}$ (clearly $dB_1 = dB_{-1}$). Observe that 
\begin{align*}
S_pB_1 & = (1,1)_p \cdot [(1,a)_p]^2 \cdot [(1,b)_p]^2 \cdot (a, b)_p = (a, b)_p\cr
S_pB_{-1} & = (-1,-1)_p \cdot [(-1,a)_p]^2 \cdot [(-1,b)_p]^2 \cdot (a, b)_p =   (-1,-1)_p \cdot (a, b)_p
\end{align*}
The result now follows from 
$$(-1, -1)_p = \left\{
\begin{array}{rl}
1 & \quad ; \quad p \text{ is odd,}\cr
-1 & \quad ; \quad p=2.
\end{array}
\right.$$
\end{proof}

\subsubsection{Bilinear forms over global fields}
We assume in this section that $\mathbb F$ is a {\em global field}, i.e. either a {\em number field} (a finite extension of $\mathbb Q$) or $\mathbb F(t)$ (the {\em function field} over a finite field $\mathbb F$). Of chief interest to us will be the example of $\mathbb Q$ itself. Global fields are distinguished by the fact that their completions with respect to any non-Archimedean valuation are local fields.  

\begin{theorem}[Hasse-Minkowski Principle, \cite{Lam}, Page 170] \label{HasseMinkowski}
Let $\mathbb F$ be a global field, let $\{\mathfrak s_i\}_{i\in \mathcal I}$ be the set of all nontrivial spots on $\mathbb F$ (with $\mathcal I$ some indexing set), and let $v_i$ be a choice of a valuation from the spot $\mathfrak s_i$ for each $i\in \mathcal I$. Then a non-degenerate bilinear form $B$ over $\mathbb F$ is isotropic if and only if $B$ is isotropic over the completion $(\mathbb E,w_i)$ of $(\mathbb F,v_i)$ for every $i\in \mathcal I$. 
\end{theorem}
\begin{corollary} \label{HasseMinkovskiOverQ}
A bilinear form $B$ is isotropic over $\mathbb Q$ if and only if it is isotropic over $\mathbb R$ and over $\mathbb Q_p$ for every prime $p$. 
\end{corollary}
A non-degenerate bilinear form $B$ over $\mathbb R$ is anisotropic if and only if it is definite, either positive ($B(x,x)>0$ for all $x\ne 0$) or negative ($B(x,x)<0$ for all $x\ne 0$). Accordingly any definite form over $\mathbb Q$ is also anisotropic. 

An indefinite form $B$ over $\mathbb Q$ is isotropic if and only if it is isotropic over $\mathbb Q_p$ for all primes $p$. By Theorem \ref{IsotropicFormsOfDimension5} this is automatically the case if rank$(B)\ge 5$, leading to this useful observation (known as Meyer's Theorem). 
\begin{theorem}  \label{TheoremMeyer} 
All non-degenerate indefinite bilinear forms over $\mathbb Q$ of rank at least 5, are isotropic. 
\end{theorem} 
The cases for which verifying the anisotropy property is nontrivial, are those of indefinite forms over $\mathbb Q$ of rank no greater than 4. Theorem \ref{HasseMinkowski} and the results from Section \ref{SectionQuadraticFormsOverLocalFields} give an effective method of dealing with such forms, as illustrated by the next example.   

\begin{example}
Determine if the form $B=\langle -5\rangle \oplus \langle 3 \rangle \oplus \langle 21\rangle \oplus \langle 7\rangle$ is isotropic over $\mathbb Q$.  

For any odd prime $p\notin \{3, 5, 7\}$, $B$ is isotropic over $\mathbb Q_p$ by Propositions \ref{PropositionAnisotropicForms} and \ref{FormsOverFiniteFields}.

For $p\in \{3, 5, 7\}$ we shall use part (b) of Proposition \ref{PropositionAnisotropicForms} along with Proposition \ref{FormsOverFiniteFields}. Specifically, for each $p\in \{3, 5, 7\}$ we shall write $B$ as a sum $B=B_1\oplus B_2$ with $B_1=\langle u_1\rangle \oplus \dots \oplus \langle u_r\rangle$ and with  $B_2=\langle pu_{r+1}\rangle \oplus \dots \oplus \langle pu_n\rangle$ with $u_i\in U(\mathbb Q_p)$ (and with $B_1$ and $B_2$ varying with $p$). Let $\bar B_1 = \langle \bar u_1\rangle \oplus \dots \oplus \langle \bar u_r\rangle$ and $\bar B_2=\langle \bar u_{r+1}\rangle \oplus \dots \oplus \langle \bar u_n\rangle$ be the corresponding forms over $\mathbb F_p$. If either $\bar B_1$ or $\bar B_2$ is isotropic over $\mathbb F_p$ then $B$ is also isotropic over $\mathbb Q_p$. 

\begin{itemize}
\item[(i)] For $p=5$ we obtain 
\begin{align*}
\bar B_1 & = \langle 3\rangle \oplus \langle 21\rangle \oplus \langle 7\rangle \cong   \langle 2\rangle \oplus \langle 1\rangle \oplus \langle 2\rangle \cr
\bar B_2 & = \langle -1\rangle .
\end{align*}
Of these $\bar B_1$ is isotropic by Proposition \ref{FormsOverFiniteFields} and so $B$ is isotropic over $\mathbb Q_5$.
\item[(ii)] For $p=7$ one finds 
\begin{align*}
\bar B_1 & = \langle -5\rangle \oplus \langle 3\rangle \oplus \langle -31\rangle \cong   \langle 1\rangle \oplus \langle -1\rangle \cr
\bar B_2 & = \langle 3\rangle \oplus \langle 1\rangle \cong \langle -1\rangle \oplus \langle 1\rangle.
\end{align*}
Both $\bar B_1$ and $\bar B_2$ are isotropic, so $q$ is isotropic over $\mathbb Q_7$.  
\item[(iii)] We left the case $p=3$ for last as 
\begin{align*}
\bar B_1 & = \langle -5\rangle \oplus \langle 7\rangle  \cong   \langle 1\rangle \oplus \langle 1\rangle \cr
\bar B_2 & = \langle 1\rangle \oplus \langle 7 \rangle  \cong   \langle 1\rangle \oplus \langle 1\rangle
\end{align*}
both of which are anisotropic over $\mathbb F_3$. It follows that $B$ is anisotropic over $\mathbb Q_3$ and therefore also over $\mathbb Q$.   
\end{itemize}
If $B$ had turned out to be isotropic over $\mathbb Q_3$ as well, we would have had to turn to $\mathbb Q_2$ and verified $B$'s isotropy property there. Since rank$(B)=4$, Part (b) of Theorem \ref{IsotropicFormsOfDimension5} shows that $B$ is in fact isotropic over $\mathbb Q_2$. 
\end{example}
\subsection{Witt Rings}  \label{SectionOnWittRings}
Let $\mathbb F$ be any field. Our goal in this section is to organize isomorphisms classes of non-degenerate symmetric bilinear forms over $\mathbb F$ into a commutative ring $W(\mathbb F)$, called the {\em Witt ring} of $\mathbb F$. The basis for its construction in the case of nondyadic fields has been laid in Witt's celebrated work \cite{Witt}. The parts of our exposition that deal with dyadic fields lean on \cite{Szymiczek}.
\subsubsection{The Witt Ring Construction}
Fix a field $\mathbb F$ of any characteristic. Define a relation $\sim$ on the set of  non-degenerate symmetric bilinear forms over $\mathbb F$ by letting $B_1$ be in relation with $B_2$, and write $B_1 \sim B_2$, if there exist metabolic forms  $B_1'$ and $B_2'$ such that 
$$B_1 \oplus B_1' \cong B_2 \oplus B_2'.$$
We shall rely on the convention that the 0-dimensional vector space with its unique bilinear form, is both anisotropic and metabolic. With this understood, it is not hard to verify that the relation $\sim$ is an equivalence relation and we shall write $[B]$ for the equivalence class of the form $B$. 

The {\em Witt ring $W(\mathbb F)$} is the commutative ring whose underlying set is that of equivalence classes of isomorphism classes of non-degenerate symmetric bilinear forms $[B]$, equipped with the operations $\oplus$  and $\otimes$, defined by 
$$[B_1]\oplus [B_2]=[B_1\oplus B_2] \qquad \text{ and } \qquad [B_1]\otimes [B_2] = [B_1\otimes B_2].$$ 
These operations are well defined on equivalence classes. The zero element in $W(\mathbb F)$ is given by the equivalence class of any metabolic form, and the inverse $-[(V,B)]$ is given by $[(V,-B)]$. Indeed, $(V,B) \oplus (V,-B)$ is isomorphic to the direct sum of $\dim _{\mathbb F} V$ metabolic spaces, each of dimension 2 (cd. \cite{Szymiczek}, Page 141). 

\begin{theorem}[\cite{Lam}, Page 12 and \cite{Szymiczek}, Page 158] \label{TheoremOnUniqueAnisotropicRepresentatives}
Let $\mathbb F$ be any field. Then every non-degenerate bilinear form $B$ over $\mathbb F$ can be decomposed as $B\cong  B_{me} \oplus B_{an}$ with $B_{me}$ a metabolic form and with $B_{an}$ anisotropic. If $B\cong  B_{me} \oplus B_{an}$ and $B\cong B'_{me} \oplus B'_{an}$
are two such decompositions, then $B_{an} \cong B'_{an}$ and rank$(B_{me})=$rank$(B'_{me})$. If additionally char $(\mathbb F)\ne 2$, then $B_{me} \cong B'_{me}$.
\end{theorem}  
We refer to the decomposition $B\cong B_{me}\oplus B_{an}$ from the preceding theorem as the {\em Witt decomposition} of the form $B$. Theorem \ref{TheoremOnUniqueAnisotropicRepresentatives} shows that the equivalence class $[B]$ contains, up to isomorphism,  a unique anisotropic representative. Given any decomposition of $B\cong B_{me}\oplus B_{an}$ as above, clearly $[B] = [B_{an}]$ and the isomorphsim class of $B_{an}$ is uniquely determined by that of $B$. This observation, in combination with   Corollary \ref{CorollaryOnDiagonalizabilityOfAnisotropicForms} yields this useful consequence. 
\begin{corollary} \label{CorollaryWittRingIsGeneratedByOneDimensionalForms}
Let $\mathbb F$ be any field. The Witt class $[B]$ of any non-degenerate symmetric bilinear form $B$ over $\mathbb F$ contains a  representative of the form $\langle a_1 \rangle \oplus \dots \oplus \langle a_n\rangle$ for some choices of $a_1, \dots, a_n \in \mathbb F^\ast$. Said differently, $W(\mathbb F)$ is additively generated by 1-dimensional forms.  
\end{corollary}
In the case of nondyadic fields, every metabolic form $B_{me}$ is a sum of copies of the hyperbolic form $\mathbb H$ from Example \ref{ExampleHyperbolicPlane}, that is $B_{me}\cong \mathbb H^n$ for some $n\ge 0$ (cf. \cite{Szymiczek}, Pages 140--144). 
\begin{corollary} \label{CorollaryDecompositionOfIsotropicFormsOverNondyadicFields}
If $B$ is a non-degenerate symmetric bilinear form over a nondyadic field $\mathbb F$, then $B\cong \mathbb H^n \oplus B_{an}$ with the integer $n\ge 0$ (called the index of isotropy of $B$) and the isomorphism class of the anisotropic form $B_{an}$, uniquely determined by $B$. 
\end{corollary}
\subsubsection{The residue homomorphisms}
Let $(\mathbb F, v)$ be a complete discretely valuated field and let $\pi\in \mathfrak p$ be a uniformizer. Recall (cf. \eqref{QuadraticFormsOverCompleteDiscretelyValuatedFields}) that a non-degenerate bilinear form $B$ over $\mathbb F$ is isomorphic to one of the form
$$B\cong \langle u_1\rangle \oplus \dots \oplus \langle u_r\rangle \oplus \langle u_{r+1}\pi\rangle \oplus \dots \oplus \langle u_n\pi\rangle,$$
with $u_1, \dots, u_n\in U$ being units, provided that either char$(\mathbb F)\ne 2$ or that $B$ is not alternating. In particular this assertion holds for all anisotropic forms. Recall also that the map from $A\to A/\mathfrak p =\overline{\mathbb F}$ is denoted by $a\mapsto \bar a$. 

With this in mind we define two functions $\partial ^1, \partial_\pi ^2 : W(\mathbb F) \to W(\overline{\mathbb F})$, called the {\em first and second residue homomorphism} respectively. On generators of $W(\mathbb F)$ (i.e. on rank 1 forms) they are determined by  
\begin{equation} \label{EquationTheResidueHomomorphisms}
\begin{array}{rlcrl}
\partial ^1([u] ) & = [ \bar u ] &    \qquad \qquad \qquad & \partial_\pi ^2([ u ] ) & = 0, \cr 
\partial ^1([ u\pi ] ) & = 0 &   \qquad \qquad \qquad & \partial_\pi ^2([ u\pi  ] ) & = [ \bar u ],
\end{array}
\end{equation}
and they are extended to all of $W(\mathbb F)$ additively. Here $[a]$ denotes the equivalence class in $W(\mathbb F)$ of the form $\langle a \rangle$. Note that $\partial ^1$ is independent of the choice of uniformizer $\pi$, but $\partial^2_\pi$ does depend on $\pi$, as reflected in our notation. The following theorem is due to Springer and Kneser. 
\begin{theorem}[\cite{Lam}, Page 147 and \cite{MilnorHusemoller}, Page 85]
The functions $\partial_\pi ^1, \partial_\pi ^2 :W(\mathbb F) \to W(\overline{\mathbb F})$ are group homomorphisms, and if char$(\mathbb F)\ne 2$, then  
$$\partial ^1 \oplus \partial_\pi ^2:W(\mathbb F) \to W(\overline{\mathbb F})\oplus W(\overline{\mathbb F})$$ 
is a group isomorphism.  
\end{theorem}
%
\subsubsection{The Witt ring $W(\mathbb Q)$}
For a prime $p$, consider the uniformizer $\pi = p$ of $\mathbb Q_p$ and let $\partial^2_p :W(\mathbb Q_p) \to W(\mathbb F_p)$ denote the second residue homomorphism from \eqref{EquationTheResidueHomomorphisms}. The inclusion $\iota_p:\mathbb Q \to \mathbb Q_p$ allows us to view a bilinear form $B$ over $\mathbb Q$ as a form over $\mathbb Q_p$, leading to a map of the same name $\iota_p:W(\mathbb Q) \to W(\mathbb Q_p)$. We define a third map $\partial_p :W(\mathbb Q) \to W(\mathbb F_p)$ as the composition of the first two, that is  $\partial _p = \partial_p^2\circ \iota_p$, and we let 
\begin{equation} \label{EquationTheMapDel}
\partial:W(\mathbb Q)  \to \bigoplus _{p \text{ prime}} W(\mathbb F_p), \qquad \qquad \partial = \oplus _p \partial _p, 
\end{equation}
denote the direct sum of all the maps $\partial _p$. With this understood, the next theorem explicitly describes $W(\mathbb Q)$. 
\begin{theorem}[\cite{Lam}, Page 175] \label{TheoremSplitExactSequenceForWittGroupOfRationals}
The sequence 
\begin{equation} \label{SESForWQ}
0\to \mathbb Z\stackrel{\iota}{\longrightarrow} W(\mathbb Q) \stackrel{\partial}{\longrightarrow} \left(\bigoplus _{p \text{ prime}} W(\mathbb F_p) \right) \to 0
\end{equation}
with $\iota (n) = n\cdot [ 1 ]$ and with $\partial$ as in \eqref{EquationTheMapDel}, is split exact. A splitting map $\sigma:W(\mathbb Q) \to \mathbb Z$ is given by the signature function $\sigma$. 
\end{theorem}
This completely determines $W(\mathbb Q)$ as 
\begin{equation} \label{EquationDeterminingWOfQ}
W(\mathbb Q) \cong \mathbb Z\oplus   \left(\bigoplus _{p \text{ prime}} W(\mathbb F_p) \right)
\end{equation}
with the map from $W(\mathbb Q)$ onto the $\mathbb Z$-summand given by the signature function. The Witt rings of the fields $\mathbb F_p$ are well understood. 
\begin{theorem}
Let $p$ be a prime integer, then there are isomorphisms of Abelian groups:
$$W(\mathbb F_p) \cong \left\{
\begin{array}{cl}
\mathbb F_2 & \quad ; \quad p=2 \cr
\mathbb F_2 \oplus \mathbb F_2 & \quad ; \quad p\equiv 1 \pmod 4,  \cr
\mathbb F_4 & \quad ; \quad p\equiv 3 \pmod 4
\end{array}
\right.$$
The generators of the groups on the right hand sides are given by $[1]$ in the case of $p=2$ or $p\equiv 3 \pmod 4$, and by $[1]$ and $[\beta ]$ in the case of $p\equiv 1 \pmod 4$ (where $\beta \in \mathbb F_p^\ast$ is any non-square). 
\end{theorem}
\begin{remark}
When working with a prime $p\equiv 1 \pmod 4$, we prefer to pick $\beta$ to be the least non-square in $\mathbb F_p$. For example we pick $\beta = 2$ if $p=5$, or $\beta = 3$ if $p=41$.
\end{remark}
\begin{example}
Let $B=\langle -15\rangle \oplus \langle 35\rangle \oplus \langle 21\rangle$.  Under the isomorphism \eqref{EquationDeterminingWOfQ}, $[B]$ maps to $1\in \mathbb Z$ and to $0 \in W(\mathbb F_p)$ for all $p\notin \{3,5, 7\}$. Additionally,  
\begin{align*}
\partial _3([B]) & = [-5] \oplus [7] = [1] \oplus [1]  \in W(\mathbb F_3), \cr   
\partial _5([B]) & = [-3] \oplus [7] = [ 2] \oplus [2] = 0 \in W(\mathbb F_5), \cr
\partial _7([B]) & = [5]\oplus [3] = [1] \oplus[1] \in W(\mathbb F_7).
\end{align*}
\end{example}
\section{Proofs} \label{SectionOnProofs}
\subsection{The Witt class of a knot} \label{SetcionWittClassOfAKnot}
Let $K$ be an oriented knot and let $S\subset S^3$ be a Seifert surface for $K$. We choose to think of the orientation on $S$ as given by a normal vector field along $S$. Associated to $S$ is the bilinear form 
$$V_S:H_1(S;\mathbb Z)\times H_1(S;\mathbb Z) \to \mathbb Z$$
defined by 
$$V_S(\alpha,\beta) = \ell k (a, b^+),$$
where $a$ and $b$ are curves on $S$ representing the homology classes $\alpha$ and $\beta$ respectively, and $b^+$ is a small push-off of $b$ from $S$ in the positive normal direction on $S$ given by its orientation, performed so that $a$ and $b^+$ are disjoint. Lastly, $\ell k (a, b^+)$ denotes the linking number of $a$ and $b^+$. Let $V_S^\tau$ denote the bilinear form on $H_1(S;\mathbb Z)$ given by $V_S^\tau (\alpha,\beta) = V_S(\beta,\alpha)$, then $V_S+V_S^\tau$ is a non-degenerate symmetric bilinear form on $H_1(S;\mathbb Z)$. The equivalence class in $W(\mathbb Q)$ of the extension of this form to $H_1(S;\mathbb Q)$, is defined to be the {\em Witt class $W(K)$ of the knot $K$}, that is 
$$W(K) = [( H_1(S;\mathbb Q), V_S+V_S^\tau)]. $$
The fact that this is well defined and independent from $S$, follows the fact that any two Seifert surfaces for $K$ are related by a finite number of handle additions/subtractions. Each handle addition to a Seifert surface $S$ changes the form $V_S+V_S^\tau$ to  $(V_S+V_S^\tau)\oplus \mathbb H$ \cite{Murasugi}, and thus the two forms generate the same Witt class in $W(\mathbb Q)$.  
\begin{theorem} Let $K$ be an oriented knot. Then 
\begin{itemize}
\item[(i)] $W(K)$ is an invariant of algebraic concordance. 
\item[(ii)] $W(K_1\#K_2) = W(K_1) \oplus W(K_2)$. 
\item[(iii)] $W(K') = W(K)$ where $K'$ is the knot $K$ with the opposite orientation. 
\end{itemize}
\end{theorem}
\subsection{Proof of Theorem \ref{main}} \label{SectionProofOfMainTheorem}
Extending Definition \ref{WittSpan}, we define the {\em Witt span $ws([B])$} for any class $[B]\in W(\mathbb Q)$ associated to a non-degenerate, symmetric bilinear form $B$, to be 
$$ws([B]) =  \min \{ \text{rank} (B')\, |\, \partial [B'] = \partial [B]\}.$$
\begin{lemma} \label{LemmaBoundOnWSFromAboveByThree}
For every non-degenerate symmetric bilinear form $B$ over $\mathbb Q$, we obtain $ws([B])\le 3$. 
\end{lemma}
\begin{proof}
Let $A=\langle a_1\rangle \oplus \dots \oplus \langle a_n\rangle$ be a form with $[A] \in \partial^{-1}(\partial([B]))$, with $a_1,\dots, a_n\in \mathbb Q^\ast$, and with $n\ge 5$. We will show that in this case there exists a form $A'$ also with $[A']\in \partial^{-1}(\partial([B]))$ and of smaller rank. 

If $A$ is not definite, then it is isotropic by  Meyer's Theorem \ref{TheoremMeyer}. By Corollary \ref{CorollaryDecompositionOfIsotropicFormsOverNondyadicFields} we can write $A\cong \mathbb H^m \oplus A'$ with $m>0$ (since $A$ is isotropic) and with $A'$ an anisotropic form. Clearly then $[A]=[A']$ and so $\partial [A] = \partial [A']$, while rank$(A')\le $rank$(A)-2$. 

If $A$ is positive definite, let $\bar A=\langle -1\rangle \oplus A$ and if $A$ is negative definite let $\bar A =  \langle 1\rangle \oplus A$. Then  $\partial [\bar A] = \partial [A]$, $\bar A$ is indefinite, has rank $n+1$ and is thus isotropic by the preceding discussion. Accordingly $\bar A=\mathbb H^m \oplus A'$ with $m>0$ and with $A'$ an anisotropic form. Note that rank$(A')\le $rank$(A)-1$. 

In summary, if $A$ is a form of rank 5 or greater, then there exists a form $A'$ with $\partial [A'] = \partial [A]$, and with rank$(A') \le   \text{rank}(A)-2$ if $A$ is indefinite, and rank$(A') \le   \text{rank}(A)-1$ if $A$ is definite. This shows that $ws([B])\le 4$ for every form $B$. 

Lastly, if $A=  \langle a_1\rangle \oplus \langle a_2\rangle \oplus \langle a_3\rangle \oplus \langle a_4\rangle$ for some $a_1,\dots, a_4\in \mathbb Q^\ast$, then either $A_1 = \langle 1\rangle \oplus A$ or $A_{-1} = \langle -1\rangle \oplus A$ is an indefinite form, and  $\partial [A_{\pm 1}]=\partial [A]$. By the conclusion from the preceding paragraph, we can then find a form $A'$ of rank no more than 3 with $\partial [A']=\partial [A]$, showing that $ws([B]) \le 3$ for all forms $B$. 
\end{proof}
\begin{proposition} \label{PropositionLowerBoundOnGammaC}
Let $K$ be a knot, $\gamma_c(K)$ its concordance crosscap number and $ws(K)$ its Witt span. Then $ws(K) \le \gamma_c(K)$. 
\end{proposition}
\begin{proof}
The proof of the claim leans heavily on the results from \cite{GordonLitherland}, where Gordon and Litherland define a bilinear form $\mathcal G_S:H_1(S;\mathbb Z)\times H_1(S;\mathbb Z) \to \mathbb Z$ associated to any spanning surface $S$ of $K$ which may be orientable or not, but in any case is non-oriented. To define $\mathcal G_S$, pick $\alpha, \beta \in H_1(S;\mathbb Z)$ and let $a, b \subset S$ be cycles representing $\alpha$ and $\beta$ respectively. Push $2b$ off of $S$ and into $S^3-S$ to obtain a new curve, call it $b'$. Then $\mathcal G_S(\alpha, \beta)$ is defined to be the linking number of $a$ with $b'$. 

When $S$ is a Seifert surface for $K$, this form recovers the symmetrized linking pairing $V_S+V_S^\tau$ from Section \ref{SetcionWittClassOfAKnot}, and by definition, the Witt class $W(K)$ of $K$ is the equivalence class of $[V_S+V_S^\tau]$ in $W(\mathbb Q)$. 

On the other hand, Gordon and Litherland prove that any two spanning surfaces for $K$ are related by a finite number of 3 types of moves  and their inverses (\cite{GordonLitherland}, Theorem 11):
\begin{itemize}
\item[(i)] Ambient isotopy. 
\item[(ii)] Addition of a 1-handle. 
\item[(iii)] Addition of a half-twisted handle (see Figure \ref{FigureHalfTwistedBands}). 
\end{itemize} 
%
\begin{figure} 
\includegraphics[width=10cm]{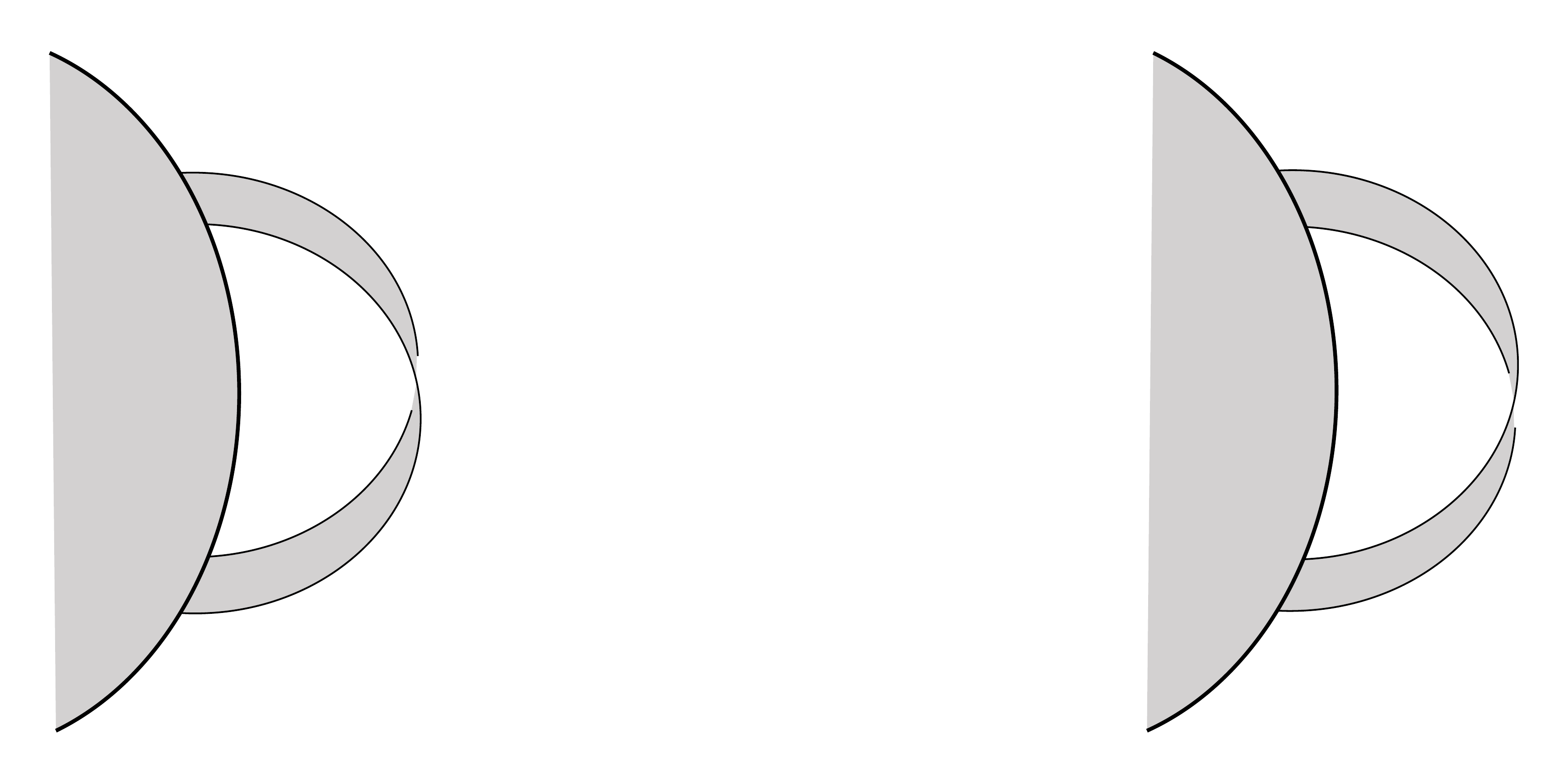}
\put(-70,70){$S$}
\put(-270,70){$S$}
\caption{Adding two types of half-twisted bands to a spanning surface $S$ of the knot $K$. } \label{FigureHalfTwistedBands}
\end{figure}
Surfaces related in this manner are said to be {\em $S^\ast$-equivalent}, whereas the usual $S$-equivalence of Seifert surfaces corresponds to modifications by moves (i) and (ii) and their inverses. The impact of either move (i)--(iii) on the pairing $\mathcal G_S$ is readily  worked out, and yields the algebraic counterparts of the moves of (i)--(iii). To more easily state them, we work with a matrix representative $G = G_S$ of $\mathcal G_S$ with respect to some basis of $H_1(S;\mathbb Z)$. We let $G'$ be the newly obtained matrix after applying each of the moves (i) -- (iii)  to $G$. Then 
\begin{itemize}
\item[(i')] $G\mapsto G'=AGA^\tau$, with $A$ an integral matrix of determinant $\pm 1$. 
\vskip3mm
\item[(ii')] $G\mapsto G'=\left[\begin{array}{ccc|cc}   
 & & & \ast & 0 \cr
& G & & \vdots & \vdots \cr
& & & \ast & 0 \cr \hline 
\ast & \dots & \ast & 0 & 1 \cr
0 & \dots & 0 & 1 & 0
 \end{array} \right]. $
\vskip3mm
\item [(iii')] $G\mapsto G'=\left[\begin{array}{c|c}   G & 0 \cr \hline  0 & \pm 1 \end{array} \right]$
\end{itemize}
It remains to observe the effect each of the algebraic moves (i')--(iii') has on the rational Witt class $[G]$ as we replace it by the rational Witt class of $[G']$. If $G'$ is obtained from $G$ via move (i'), then $G$ and $G'$ are isomorphic and so $[G]=[G']$. If $G'$ is obtained from $G$ by a move of type (ii') then $G' \cong G\oplus \mathbb H$ and thus again $[G] = [G']$. Lastly, if $G'$ arises from $G$ via move (iii'), then $[G'] = [G] \oplus [ \pm 1]$. This move certainly changes the Witt class, but it does so by an element in the kernel of the homomorphism $\partial :W(\mathbb Q) \to \oplus _{p \text{ prime}} W(\mathbb F_p)$ from \eqref{SESForWQ}. We conclude that $\partial ([G])$ is independent of the surface $S$ used to compute it. 

If $S$ is a nonorientable spanning surface for $K$ realizing $\gamma_3(K)$, then the associated form $\mathcal G_S$ has rank equal to $\gamma_3(K)$ and its equivalence class in $W(\mathbb Q)$ lies in $\partial ^{-1}(W(K))$. Since $ws(K)$ is the minimum rank of all such forms, it follows that $ws(K)  \le \gamma_3(K)$. Lastly, since the Witt class $W(K)$ is a concordance invariant, it follows that so is $ws(K)$ implying that $ws(K) \le \gamma_c(K)$, as claimed.    
\end{proof}
Lemma \ref{LemmaBoundOnWSFromAboveByThree}, Proposition \ref{PropositionLowerBoundOnGammaC} and Example \ref{ExampleGivingFullRangeOfWittSpan} below, complete the proof of Theorem \ref{main}. 
\subsection{Algorithm for computing the Witt span}
This section outlines how to compute $ws([B])$ for any non-degenerate symmetric bilinear form $B$ over $\mathbb Q$. Recall that $ws([B])$ is the minimal rank among all forms $A$ with $[A]\in \partial ^{-1}(\partial ([B]))$ where $\partial :W(\mathbb Q) \to \oplus _{p \text{ prime}}W(\mathbb F_p)$ is as in \eqref{SESForWQ}. 

For a bilinear form $B$, the set $\partial^{-1}(\partial ([B]))$ can be explicitly described courtesy of Theorem \ref{TheoremSplitExactSequenceForWittGroupOfRationals}:
$$\partial^{-1}(\partial([B])) = \{ (a \cdot [1]) \oplus (b\cdot [-1]) \oplus [B] \, |\, |a, b \ge 0\}.$$
The isomorphisms types of the forms $A$ whose Witt classes belong to this set are given by 
$$[A]\in \partial^{-1}(\partial([B])) \qquad \Longleftrightarrow  \qquad  A\cong (a \cdot \langle 1\rangle ) \oplus (b\cdot \langle -1\rangle ) \oplus B\oplus \mathbb H^m,$$
with $a, b, m \ge 0$. Let  $B\cong \mathbb H^n\oplus B_{an}$ be the Witt decomposition for $B$ with $n\ge 0$ and $B_{an}$ anisotropic, and recall that over $\mathbb Q$, we have $\mathbb H\cong \langle 1 \rangle \oplus \langle -1 \rangle$. Thus we the arrive at this characterization of forms $A$ with $[A] \in \partial ^{-1}(\partial (B))$:
\begin{equation} \label{EquationListOfFormToFindWittSpan}
[A]\in \partial^{-1}(\partial([B])) \qquad \Longleftrightarrow \qquad  A\cong (a \cdot \langle 1\rangle ) \oplus (b\cdot \langle -1\rangle ) \oplus B_{an}, \quad a, b, \ge 0.
\end{equation}
The task then is to find the Witt decompositions (see sentence after Theorem \ref{TheoremOnUniqueAnisotropicRepresentatives}) of the forms in \eqref{EquationListOfFormToFindWittSpan} and determine the smallest rank of the anisotropic parts of said decompositions, that rank then equals $ws([B])$. Without loss of generality, we can assume at the outset that rank$(B_{an})\le 3$, for if rank$(B_{an})\ge 4$, then at least one of $B_{an}\oplus \langle \pm 1\rangle$ is indefinite and of rank at least 5, and hence isotropic by Meyer's Theorem \ref{TheoremMeyer}. As such the form is isomorphic to $B'_{an}\oplus \mathbb H$ with $B'_{an}$ an anisotropic form of rank one less than the rank of $B_{an}$ and with $\partial [B'_{an}] = \partial [B_{an}]$ (the fact that $\mathbb H$ shows up with power 1 in this decomposition, relies on \eqref{EquationOnTheChangeOfIsotropyIndex} below). 
\begin{theorem} \label{TheoremAlgorithmForComputingws}
Let $B$ be a non-degenerate, symmetric, bilinear form over $\mathbb Q$. For $r=1, 2, 3$ let $\mathcal I_r$ be the set of signatures occurring among rank $r$, anisotropic forms whose Witt equivalence class lies in $\partial ^{-1}(\partial ([B]))$.  Then
\begin{itemize}
\item[(i)] $ws([B]) = 0$ if and only if $\partial ([B]) = 0$.   
\item[(ii)] $ws([B])=1$ if and only if $\mathcal I_1\ne \emptyset$ and $\partial ([B]) \ne 0$.
\item[(iii)]  $ws([B]) = 2$ if and only if $\mathcal I_2 \ne \emptyset$ and for each $i\in \mathcal I_2$, the forms $B_i\oplus \langle -1 \rangle$ and $B_i\oplus \langle -1 \rangle$ are anisotropic. The set $I_2$ can equal any non-empty subset of $\{0, \pm 2\}$. 
\item[(iv)] $ws([B])=3$ if and only if $\mathcal I_3 = \{\pm 1, \pm 3\}$ and the forms $B_{i}\oplus \langle 1 \rangle$ and $B_{i}\oplus \langle -1 \rangle$ are anisotropic for each $i\in \mathcal I_3$.
\end{itemize}

\end{theorem}
\begin{proof}
In what follows we shall make repeated tacit use of the following formula (cf. \cite{Szymiczek}, Page 151) applicable to any pair of forms $B_{an}, C$ of which $B_{an}$ is anisotropic:
\begin{equation} \label{EquationOnTheChangeOfIsotropyIndex}
\text{ind} ( B_{an} \oplus C) + \text{ind}(C) \le \text{rank}(C).
\end{equation}
Here ``ind" refers to the index of isotropy of a form as defined in Corollary \ref{CorollaryDecompositionOfIsotropicFormsOverNondyadicFields}. In particular, if $C=\langle \pm 1 \rangle$ then 
ind$(B_{an}\oplus \langle \pm 1 \rangle ) \le 1$. This allows for two alternatives, either $B_{an}\oplus \langle \pm 1 \rangle$ is anisotropic or  $B_{an}\oplus \langle \pm 1 \rangle \cong B'_{an}\oplus \mathbb H$ with $B'_{an}$ another anisotropic form. Said differently, adding $\langle \pm 1\rangle$ to an anisotropic form can generate at most one hyperbolic summand $\mathbb H$ in its Witt decompostion. With this understood, we turn to the proof of the theorem, skipping over cases (i) and (ii) which are obvious. From hereon out, let $B\cong B_{an}\oplus \mathbb H^n$ be the Witt decomposition of $B$ as in \eqref{EquationListOfFormToFindWittSpan}. 
\vskip3mm
\noindent Case (iii). If $ws([B]) = 2$, there must exist a rank 2 form with Witt equivalence class in $\partial^{-1}(\partial ([B]))$ and thus $\mathcal I_2\ne \emptyset$. If for any $i\in \mathcal I_2$, either $B_i \oplus \langle -1 \rangle $ or $B_i \oplus \langle 1 \rangle$ were isotropic, we could write that form as $B'\oplus \mathbb H$ with $B'$ an anisotropic form of rank 1 and with $\partial [B']=\partial [B]$, a contradiction to the assumption of $ws([B]) = 2$. 

Conversely, suppose that $\mathcal I_2\ne \emptyset $ and that for each $i\in \mathcal I_2$, both forms $B_i \oplus \langle \pm 1\rangle$ are anisotropic. Without loss of generality we may assume that $B_{an} \cong B_i$ for some choice of $i\in \mathcal I_2$. We next find the Witt decomposition of all forms from \eqref{EquationListOfFormToFindWittSpan}, by considering various possible cases for $\mathcal I_2$. 

If $B_{an} \cong B_2$, then by assumption $B_2 \oplus \langle \pm 1 \rangle$ is anisotropic, but $B_2 \oplus \langle -1 \rangle\oplus \langle - 1 \rangle$ may be isotropic or anisotropic, splitting the discussion into two subcases, both of which occur in practice (note that there is no need to consider $B_2\oplus \langle 1 \rangle \oplus \langle 1 \rangle$ as this form is positive definite and hence anisotropic):
\begin{itemize}
\item[(a)] Suppose $B_2\oplus \langle -1 \rangle \oplus \langle -1 \rangle $ is isotropic and write  
$$B_2\oplus \langle -1 \rangle \oplus \langle -1 \rangle  \cong B_0 \oplus \mathbb H$$
for an anisotropic form $B_0$ or rank 2 and signature 0, forcing $0\in \mathcal I_2$. By assumption $B_0\oplus \langle \pm 1\rangle$ is anisotropic, but $B_0\oplus \langle -1 \rangle  \oplus \langle -1 \rangle $ may be either isotropic or anisotropic, splitting the discussion into two additional subcases. 
\begin{itemize}
\item[(a1)] Suppose that $B_0\oplus \langle -1 \rangle  \oplus \langle -1 \rangle $ is isotropic and write 
$$B_0\oplus \langle -1 \rangle  \oplus \langle -1 \rangle  \cong B_{-2}\oplus \mathbb H$$
for an anisotropic form $B_{-2}$ of rank 2 and signature $-2$. Note that then $-2 \in \mathcal I_2$ and thus $\mathcal I_2 = \{0, \pm 2\}$. 
\item[(a2)] Suppose that $B_0\oplus \langle -1 \rangle  \oplus \langle -1 \rangle $ is anisotropic, in which case  $B_0\oplus \langle -1 \rangle  \oplus \langle -1 \rangle \oplus \langle -1 \rangle $ is isotropic (by Meyer's Theorem) and so 
$$B_0\oplus \langle -1 \rangle  \oplus \langle -1 \rangle \oplus \langle -1 \rangle \cong B_{-3}\oplus \mathbb H.$$
Here $B_{-3}$ is an anisotropic form or rank 3 and signature $-3$. We shall see in a bit that in the present case one obtains $\mathcal I_2 = \{0,2\}$
\end{itemize}
\item[(b)] Suppose $B_2\oplus \langle -1 \rangle \oplus \langle -1 \rangle $ is anisotropic, then  $B_2\oplus \langle -1 \rangle \oplus \langle -1 \rangle \oplus \langle -1 \rangle$ must be isotropic, and we can write 
$$B_2\oplus \langle -1 \rangle \oplus \langle -1 \rangle \oplus \langle -1 \rangle \cong B_{-1}\oplus \mathbb H$$
for an anisotropic form $B_{-1}$ of rank $3$ and signature $-1$. The form $B_{-1}\oplus \langle -1 \rangle$ may be isotropic or anisotropic. 
\begin{itemize}
\item[(b1)] If $B_{-1}\oplus \langle -1 \rangle$ is isotropic then 
$$B_{-1}\oplus \langle -1 \rangle \cong B_{-2}\oplus \mathbb H$$
for some rank 2, anisotropic form $B_{-2}$ of signature $-2$. As we shall see below, in this case $\mathcal I_2 = \{\pm 2\}$. 
\item[(b2)] If $B_{-1}\oplus \langle -1 \rangle$ is anisotropic then $B_{-1}\oplus \langle -1 \rangle \oplus \langle -1 \rangle $ is isotropic and thus 
$$ B_{-1}\oplus \langle -1 \rangle \oplus \langle -1 \rangle \cong B_{-3}\oplus \mathbb H.$$
Here $B_{-3}$ is a rank 3 form of signature $-3$. We shall see that in the present case one finds $\mathcal I_2 = \{2 \}$.
\end{itemize}
\end{itemize}
Thus, with the assumption of $2\in \mathcal I_2$, the Witt decompositions of the forms \eqref{EquationListOfFormToFindWittSpan} with $B_{an} = B_2$, become:

\begin{align*}
(a1) \Longrightarrow (b\cdot \langle -1 \rangle ) \oplus B_2 & \cong \left\{
\begin{array}{rl}
((b\cdot \langle -1 \rangle ) \oplus B_2) \oplus \mathbb H^0 & \quad ; \quad b=0, 1 \cr
(((b-2)\cdot \langle -1 \rangle ) \oplus B_0) \oplus \mathbb H^1 & \quad ; \quad b=2, 3 \cr
(((b-4)\cdot \langle -1 \rangle ) \oplus B_{-2}) \oplus \mathbb H^1 & \quad ; \quad b\ge 4 \cr
\end{array}
\right. \cr \cr
(a2) \Longrightarrow (b\cdot \langle -1 \rangle ) \oplus B_2 & \cong \left\{
\begin{array}{rl}
((b\cdot \langle -1 \rangle ) \oplus B_2) \oplus \mathbb H^0 & \quad ; \quad b=0, 1 \cr
(((b-2)\cdot \langle -1 \rangle ) \oplus B_0) \oplus \mathbb H^1 & \quad ; \quad b=2, 3, 4 \cr
(((b-5)\cdot \langle -1 \rangle ) \oplus B_{-3}) \oplus \mathbb H^1 & \quad ; \quad b\ge 5 \cr
\end{array}
\right. \cr \cr
(b1) \Longrightarrow (b\cdot \langle -1 \rangle ) \oplus B_2 & \cong \left\{
\begin{array}{rl}
((b\cdot \langle -1 \rangle ) \oplus B_2) \oplus \mathbb H^0 & \quad ; \quad b=0, 1, 2 \cr
(((b-3)\cdot \langle -1 \rangle ) \oplus B_{-1}) \oplus \mathbb H^1 & \quad ; \quad b=3, \cr
(((b-4)\cdot \langle -1 \rangle ) \oplus B_{-2}) \oplus \mathbb H^1 & \quad ; \quad b\ge 4 \cr
\end{array}
\right. \cr \cr
(b2) \Longrightarrow (b\cdot \langle -1 \rangle ) \oplus B_2 & \cong \left\{
\begin{array}{rl}
((b\cdot \langle -1 \rangle ) \oplus B_2) \oplus \mathbb H^0 & \quad ; \quad b=0, 1, 2 \cr
(((b-3)\cdot \langle -1 \rangle ) \oplus B_{-1}) \oplus \mathbb H^1 & \quad ; \quad b=3, 4 \cr
(((b-4)\cdot \langle -1 \rangle ) \oplus B_{-3}) \oplus \mathbb H^1 & \quad ; \quad b\ge 5 \cr
\end{array}
\right. \cr \cr
\end{align*}
From these we can read off $\mathcal I_2$ in the various cases:
$$\mathcal I_2 = \left\{
\begin{array}{cl}
\{0,\pm 2\} & \quad ; \quad \text{Case (a1)}, \cr
\{0,2\} & \quad ; \quad \text{Case (a2)}, \cr
\{\pm 2\} & \quad ; \quad \text{Case (b1)}, \cr
\{2\} & \quad ; \quad \text{Case (b2)}, \cr
\end{array}
\right.$$
By inspection, we see that the minimum rank of any anisotropic component of the forms $(a\cdot \langle 1 \rangle ) \oplus (b\cdot \langle -1 \rangle) \oplus B_{an}$ is 2, proving the theorem for the case of $B_{an} \cong B_2$. The case of $B_{an}\cong B_{-2}$ proceeds in complete analogy, chiefly by changing signs, we omit the details. Instead, we turn to the case of $B_{an}\cong B_0$.  By assumption, $B_0\oplus \langle \pm 1 \rangle $ are anisotropic, but $B_0\oplus \langle \pm 1 \rangle \oplus \langle \pm 1 \rangle$ may be isotropic or anisotropic. 
\begin{itemize}
\item[(A)] If $B_0\oplus \langle 1 \rangle \oplus \langle 1 \rangle$  is isotropic, we may write 
$$B_0\oplus \langle 1 \rangle \oplus \langle 1 \rangle  \cong B_2\oplus \mathbb H$$
for an anisotropic form $B_2$ or rank and signature both equal to 2. 
\item[(B)] If $B_0\oplus \langle -1 \rangle \oplus \langle -1 \rangle$  is isotropic, we may write 
$$B_0\oplus \langle -1 \rangle \oplus \langle -1 \rangle  \cong B_{-2}\oplus \mathbb H$$
for an anisotropic form $B_{-2}$ or rank 2 and signature $-2$.  
\item[(C)] If $B_0\oplus \langle 1 \rangle \oplus \langle 1 \rangle$ is anisotropic, then If $B_0\oplus \langle 1 \rangle \oplus \langle 1 \rangle \oplus \langle 1 \rangle$  is isotropic and so 
$$B_0\oplus \langle 1 \rangle \oplus \langle 1 \rangle \oplus \langle 1 \rangle \cong B_3\oplus \mathbb H.$$
Here $B_3$ has rank and signature 3, and it is anisotropic.   
\item[(D)] If $B_0\oplus \langle -1 \rangle \oplus \langle -1 \rangle$ is anisotropic, then If $B_0\oplus \langle -1 \rangle \oplus \langle -1 \rangle \oplus \langle -1 \rangle$  is isotropic and so 
$$B_0\oplus \langle -1 \rangle \oplus \langle -1 \rangle \oplus \langle -1 \rangle \cong B_{-3}\oplus \mathbb H.$$
Here $B_3$ has rank 3 and signature $-3$, and it is anisotropic.   
\end{itemize}
With these in place we can find the Witt decompositions of the forms from \eqref{EquationListOfFormToFindWittSpan} with $B_{an} = B_0$ as: 
\begin{align*}
(A) \Longrightarrow (a\cdot \langle 1 \rangle ) \oplus B_0 \cong \left\{
\begin{array}{rl}
((a\cdot\langle 1 \rangle ) \oplus B_0 )\oplus \mathbb H^0 & \quad ; \quad a=0, 1, \cr
(((a-2)\cdot\langle 1 \rangle ) \oplus B_2 )\oplus \mathbb H^1 & \quad ; \quad a\ge 2.
\end{array}
\right. \cr \cr
(B) \Longrightarrow (b\cdot \langle -1 \rangle ) \oplus B_0 \cong \left\{
\begin{array}{rl}
((b\cdot\langle 1 \rangle ) \oplus B_0 )\oplus \mathbb H^0 & \quad ; \quad b=0, 1, \cr
(((b-2)\cdot\langle 1 \rangle ) \oplus B_{-2} )\oplus \mathbb H^1 & \quad ; \quad b\ge 2.
\end{array}
\right. \cr \cr
(C) \Longrightarrow (a\cdot \langle 1 \rangle ) \oplus B_0 \cong \left\{
\begin{array}{rl}
((a\cdot\langle 1 \rangle ) \oplus B_0 )\oplus \mathbb H^0 & \quad ; \quad a=0, 1, 2\cr
(((a-3)\cdot\langle 1 \rangle ) \oplus B_3 )\oplus \mathbb H^1 & \quad ; \quad a\ge 3.
\end{array}
\right. \cr \cr
(D) \Longrightarrow (b\cdot \langle -1 \rangle ) \oplus B_0 \cong \left\{
\begin{array}{rl}
((b\cdot\langle 1 \rangle ) \oplus B_0 )\oplus \mathbb H^0 & \quad ; \quad b=0, 1, 2\cr
(((b-3)\cdot\langle 1 \rangle ) \oplus B_{-3} )\oplus \mathbb H^1 & \quad ; \quad b\ge 3.
\end{array}
\right. \cr \cr
\end{align*}
From these Witt decomposition we read off the set $\mathcal I_2$:
$$\mathcal I_2 = \left\{
\begin{array}{cl}
\{0,\pm 2\} & \quad ; \quad \text{Cases (A) \& (B) occurred}, \cr
\{0,2\} & \quad ; \quad \text{Cases (A) \& (D) occurred}, \cr
\{0,-2\} & \quad ; \quad \text{Cases (C) \& (B) occurred}, \cr
\{0\} & \quad ; \quad \text{Cases (C) \& (D) occurred}.
\end{array}
\right.$$
This completes the proof of case (iii) of the theorem. It is worthwhile pointing out that cases (A) and (B) above (and similarly cases (C) and (D)) ``almost" coincide, but not quite. Indeed, since $B_0$ is a signature 0 form, each of $B_0\oplus \langle \pm 1\rangle \oplus \langle \pm 1\rangle $ is isotropic over $\mathbb R$. Thus, to be anisotropic over $\mathbb Q$, $B_0\oplus \langle \pm 1\rangle \oplus \langle \pm 1\rangle $ must be aninsotropic over $\mathbb Q_p$ for some prime $p$. If such a $p$ can be found among odd primes, then $B_0\oplus \langle 1\rangle \oplus \langle 1 \rangle $ is anisotropic over $\mathbb Q_p$ if and only if $B_0\oplus \langle -1\rangle \oplus \langle -1 \rangle $ is anisotropic over $\mathbb Q_p$, cf. Corollary \ref{CorollaryAboutAnisotropyOverQ2OfCertainRank4Forms}. However, such an odd prime $p$ may not always exist. For instance, for $B_0=\langle -3\rangle \oplus \langle 5\rangle$ we find that $B_0\oplus \langle \pm 1 \rangle \oplus \langle \pm 1\rangle$ is isotropic over $\mathbb Q_p$ for every odd prime $p$, but that $B_0\oplus \langle 1 \rangle \oplus \langle 1\rangle $ is anisotropic over $\mathbb Q_2$ and thus anisotropic over $\mathbb Q$, while $B_0\oplus \langle -1 \rangle \oplus \langle -1\rangle $ is isotropic over $\mathbb Q_2$ and thus also over $\mathbb Q$. This shows that the cases (A) and (B) (and likewise (C) and (D)) must be considered as separate cases. 
\vskip3mm
\noindent Case (iv).  
If $ws([B])=3$ there must exist a rank 3 anisotropic form $B_i$ with $[B_i]\in \partial^{-1}(\partial ([B]))$ with $\sigma (B_i) = i$ for some choice of $i\in \{\pm 1, \pm 3\}$. If either of $B_i\oplus \langle 1 \rangle $ or $B_i\oplus \langle -1 \rangle$ were isotropic, the said form would be isomorphic to $\mathbb H\oplus B'$ with $B'$ an anisotropic form of rank 2 and with $\partial [B'] = \partial [B]$, contradicting the assumption of $ws([B]) = 3$. If $i\in \{\pm 1\}$, then both forms $B_i\oplus \langle 1 \rangle \oplus \langle 1 \rangle $ or $B_i\oplus \langle -1 \rangle \oplus \langle -1 \rangle$ must be isotropic (being or rank 5 and indefinite, see Meyer's Theorem \ref{TheoremMeyer}), and thus 
$$B_i\oplus \langle -1 \rangle \oplus \langle -1 \rangle \cong B_{i-2}\oplus \mathbb H \qquad \text{ and } \qquad B_i\oplus \langle 1 \rangle \oplus \langle 1 \rangle \cong B_{i+2}\oplus \mathbb H.$$
In the above, $B_{i\pm 2}$ are anisotropic forms of rank 3 with $\sigma (B_{i\pm 2}) = i\pm 2$ and $\partial [B_{i\pm 2}]=\partial [B]$. Thus, if $i=1$ these constructions yield the forms $B_3$ and $B_{-1}$, while if $i=-1$, the constructions yield $B_1$ and $B_{-3}$. In either case we have obtained both $B_{\pm 1}$ and therefore also $B_{\pm 3}$. If the initial form $B_i$ had $i\in \{\pm 3\}$, the same construction works to firstly yield $B_{\pm 1}$ via 
$$B_3\oplus \langle -1 \rangle \oplus \langle -1 \rangle \cong B_1 \oplus \mathbb H \qquad \text{ or } \qquad  B_{-3}\oplus \langle 1 \rangle \oplus \langle 1 \rangle \cong B_{-1} \oplus \mathbb H,$$
and then proceed as before. In summary, if $ws([B])=3$, all four forms $B_i$, $i\in \{\pm 1, \pm 3\}$ must exist and each of  $B_i \oplus \langle -1 \rangle$ and $B_i \oplus \langle 1 \rangle$ is anisotropic. 

Conversely, suppose there exist rank 3, anisotropic forms $B_i$ with $[B_i] \in \partial ^{-1}(\partial ([B]))$, $i\in \{\pm 1, \pm 3\}$, with $\sigma (B_i) = i$, and with each of $B_i \oplus \langle -1 \rangle$ and $B_i \oplus \langle 1 \rangle$ anisotropic. As in the preceding case, we may again without loss of generality assume that $B_{an} \cong B_i$ for some $i\in \{\pm 1, \pm 3\}$. It remains to find the Witt decompositions of the forms \eqref{EquationListOfFormToFindWittSpan} for the various possibilities of $i\in \{\pm 1, \pm 3\}$. 

$$
(b\cdot \langle -1 \rangle ) \oplus B_3 \cong \left\{
\begin{array}{rl}
(B_3)\oplus \mathbb H^0 & \quad ; \quad b=0, \cr
(\langle -1\rangle \oplus B_3) \oplus \mathbb H^0 & \quad ; \quad b=1, \cr
B_1 \oplus \mathbb H^1  & \quad ; \quad b=2, \cr
(\langle -1\rangle \oplus B_1) \oplus \mathbb H^1  & \quad ; \quad b=3, \cr
(B_{-1} ) \oplus \mathbb H^2  & \quad ; \quad b=4, \cr
(\langle -1\rangle \oplus B_{-1}) \oplus \mathbb H^2  & \quad ; \quad b=5, \cr
((b-6)\cdot \langle -1\rangle \oplus B_{-3}) \oplus \mathbb H^3  & \quad ; \quad b\ge 6.
\end{array}
\right.
$$
The forms on the right-hand side enclosed in parenthesis represent the anisotropic parts of the Witt decomposition, and we see that among them the lowest rank is 3. The remaining cases of $i\in \{-3, \pm 1\}$ follow similarly. We make explicit the case of $i=-1$ and leave the other two cases as easy exercises for the interested reader: 
$$
(a\cdot \langle 1 \rangle ) \oplus B_{-1} \cong \left\{
\begin{array}{rl}
(B_{-1})\oplus \mathbb H^0 & \quad ; \quad a=0, \cr
(\langle 1\rangle \oplus B_{-1}) \oplus \mathbb H^0 & \quad ; \quad a=1, \cr
(B_{1}) \oplus \mathbb H^1  & \quad ; \quad a = 2, \cr 
(\langle 1 \rangle \oplus B_{1}) \oplus \mathbb H^1  & \quad ; \quad a = 3, \cr 
((a-4)\cdot \langle 1 \rangle \oplus B_{3} )  \oplus \mathbb H^2  & \quad ; \quad a \ge 4, \cr 
\end{array}
\right.
$$
$$
(b\cdot \langle -1 \rangle ) \oplus B_{-1} \cong \left\{
\begin{array}{rl}
(B_{-1})\oplus \mathbb H^0 & \quad ; \quad b=0, \cr
(\langle -1\rangle\oplus B_{-1} ) \oplus \mathbb H^0 & \quad ; \quad b=1, \cr
((b-2)\cdot \langle -1\rangle \oplus B_{-3})  \oplus \mathbb H^1  & \quad ; \quad b\ge 2. 
\end{array}
\right.
$$
Once again, by inspection, we see that the smallest rank among the anisotropic forms in these Witt decopositions equals 3. This concludes the proof of part (iv) of the theorem. 
\end{proof}
The goal of the next example is to illustrate the effectiveness of Theorem \ref{TheoremAlgorithmForComputingws} in computing the Witt span of knots, and to provide examples of knots $K_i$ with $ws(K_i) = i$ for each $i\in \{0,1,2,3\}$, thereby completing the proof of Theorem \ref{main}. 
\begin{example} \label{ExampleGivingFullRangeOfWittSpan} If $K_0$ is any slice knot, then $W(K_0)=0$ and thus $ws(K_0) = 0$. If $K_1$ is any of the alternating torus knots $T(2,p)$ with $p>0$ odd, then $W(K_1) = [ p ]$ and $\partial (W(K))\ne 0$,  showing that $ws(K_1) = 1$.  
\vskip3mm
Let $K_2=9_{40}$ as in Example \ref{Example940}, then $W(K_2)$ can be calculated from the Seifert matrix of $K_2$ listed on KnotInfo \cite{KnotInfo}:
$$W(K_2) = [-10] \oplus [-30].$$
We see that $B_{-2}=\langle -10\rangle \oplus \langle  -30\rangle$ is an anisotropic form with $[B_{-2}]\in \partial ^{-1}(\partial (W(K_2)))$, and so $ws(K_2)\le 2$. Clearly $B_{-2}\oplus \langle -1\rangle$ is anisotropic, as is the form $B_{-2}\oplus \langle 1 \rangle$ (being anisotropic over $\mathbb Q_{5}$). The form $B_{-2}\oplus \langle 1 \rangle \oplus \langle 1 \rangle$ is isotropic over $\mathbb Q$ and  
$$B_{-2}\oplus \langle 1 \rangle \oplus \langle 1 \rangle  \cong B_0 \oplus \mathbb H \qquad \text{ with } \qquad B_0 \cong \langle 10 \rangle \oplus \langle -30 \rangle.$$%
$B_0\oplus \langle -1\rangle $ is anisotropic (being isomorphic to $B_{-2}\oplus \langle 1 \rangle$) as is $B_0\oplus \langle 1\rangle$ (being anisotropic over $\mathbb Q_3$). Lastly, 
$$B_0 \oplus \langle 1 \rangle   \oplus \langle 1 \rangle  \cong B_2 \oplus \mathbb H  \qquad \text{ with } \qquad B_2 \cong \langle 155 \rangle \oplus \langle 465 \rangle. $$
Observe that $B_2 \oplus \langle \pm 1 \rangle$ is anisotropic over $\mathbb Q$. Thus, by Theorem \ref{TheoremAlgorithmForComputingws} we find that $ws(9_{40}) = 2$ (even though $K_2=9_{40}$ does not satisfy the hypothesis of Part (i) of Theorem \ref{TheoremCharacterizingWsTwoThree}). In the notation of Theorem \ref{TheoremAlgorithmForComputingws}, we obtain $\mathcal I_2=\{-2,0,2\}$, indeed in the proof of said theorem the current example mirrors case (a1). 
\vskip3mm
By way of another example of a knot with Witt span 2, consider $K'_2=7_5$. Then $W(K'_2)$ obtained from its Seifert surface on KnotInfo \cite{KnotInfo}, is :
$$W(K'_2) = [ -1 ] \oplus [-1 ] \oplus [ -3] \oplus [ -51].$$
We see that $B_{-2}=\langle -3\rangle \oplus \langle  -51\rangle$ is an anisotropic form with $[B_{-2}]\in \partial ^{-1}(\partial (W(K'_2)))$, and so $ws(K'_2)\le 2$. Clearly $B_{-2}\oplus \langle -1\rangle$ is anisotropic, and it is easy to see that $B_{-2}\oplus \langle 1 \rangle$ is anisotropic over $\mathbb Q_{17}$ and thus also over $\mathbb Q$. The form $B_{-2}\oplus \langle 1 \rangle \oplus \langle 1 \rangle$ is anisotropic over $\mathbb Q_2$ by Theorem \ref{TheoremCharacterizingSmallRankFormsOverQp}, but $B_{-2}\oplus \langle 1 \rangle \oplus \langle 1 \rangle \oplus \langle 1 \rangle$ is isotropic and 
$$   B_{-2}\oplus \langle 1 \rangle \oplus \langle 1 \rangle \oplus \langle 1 \rangle \cong B_1 \oplus \mathbb H \qquad \text{ with } \qquad B_1 \cong \langle 14 \rangle \oplus \langle 42 \rangle \oplus \langle -51 \rangle .$$%
Additionally
$$B_1\oplus \langle 1 \rangle  \cong B_2 \oplus \mathbb H \qquad \text{ with } \qquad B_2 = \langle 750414 \rangle \oplus \langle 44142\rangle.$$
Since $B_2\oplus \langle -1\rangle$ is isomorphic to $B_1$ it is clearly anisotropic, while $B_2 \oplus \langle 1 \rangle$ is anisotropic on account of being positive definite. Thus, by Theorem \ref{TheoremAlgorithmForComputingws} we find that $ws(7_5) = 2$. In the notation of Theorem \ref{TheoremAlgorithmForComputingws}, we find that $\mathcal I_2=\{\pm 2\}$, and in the proof of said theorem the current example mirrors case (b1). 
\vskip3mm
Next consider next the knot $K_3 = 9_{49}$. Relying again on KnotInfo we obtain 
$$W(K_3) = B_{-4} = \langle -2 \rangle \oplus \langle -14 \rangle \oplus \langle -10\rangle \oplus \langle -70 \rangle.$$
As this form is isotropic over $\mathbb Q$, being negative definite, consider the isotropic form 
$$B_{-4}\oplus \langle 1 \rangle  \cong B_{-3}\oplus \mathbb H \qquad \text{ with } \qquad B_{-3} = \langle -70 \rangle\oplus \langle -10 \rangle\oplus \langle -7 \rangle.$$
Clearly $B_{-3}\oplus \langle -1 \rangle$ is anisotropic, as is $B_{-3}\oplus \langle 1 \rangle$ (it is anisotropic over $\mathbb Q_5$). Moreover 
$$ B_{-3}\oplus \langle 1 \rangle \oplus \langle 1 \rangle  \cong B_{-1} \oplus \mathbb H \qquad \text{ and } \qquad B_{-1} = \langle -70 \rangle \oplus \langle -7  \rangle \oplus \langle 10 \rangle. $$
$B_{-1}\oplus \langle -1 \rangle $ is automatically anisotropic as it is isomophic to $B_{-3}\oplus \langle 1 \rangle $, while $B_{-1}\oplus \langle 1 \rangle $ is anisotropic over $\mathbb Q_5$ and thus over $\mathbb Q$. Next we find that 
$$B_{-1} \oplus \langle 1 \rangle  \oplus \langle 1 \rangle  \cong B_1 \oplus \mathbb H \qquad \text{ with } \qquad B_1 = \langle -70 \rangle  \oplus\langle 163 \rangle  \oplus\langle 11410 \rangle .$$
As always, $B_1 \oplus \langle -1\rangle$ is automatically anisotropic, while $B_1 \oplus \langle 1 \rangle $ is anisotropic over $\mathbb Q$ because it is anisotropic over $\mathbb Q_5$. Lastly 
$$B_1 \oplus \langle 1 \rangle  \oplus \langle 1 \rangle \cong B_3 \oplus \mathbb H \qquad \text{ with } \qquad B_3 \cong \langle 11410 \rangle  \oplus\langle 156253 \rangle  \oplus\langle 1782846730 \rangle .$$ 
$B_3\oplus \langle -1\rangle \cong B_1 \oplus \langle 1 \rangle $ is anisotropic, as is $B_3\oplus \langle 1\rangle$ because it is positive definite. In summary, we have identified all four forms $B_i$, $i\in \{\pm 1, \pm 3\}$ and shown that $B_i \oplus \langle \pm 1 \rangle$ is anisotropic for each $i$. By Theorem \ref{TheoremAlgorithmForComputingws}, it follows that $ws(9_{49}) = 3$. 
\end{example}
\subsection{Preimages of the map $\partial$}
This section outlines a prescription for finding elements in the preimage of the map $\partial :W(\mathbb Q) \to \oplus _{p \text{ prime}} W(\mathbb F_p)$ from \eqref{SESForWQ}. This explicit description  can sometimes be used to substantially shorten computations, and this is one of its utilities in the context of this work. The other is its contribution to the proof ot Theorem \ref{TheoremCharacterizingWsTwoThree}.

In the following proposition, we adopt this notation: We shall write $[1_p]$ to denote the class $[ 1 ] \in W(\mathbb F_p)$, and if $p\equiv 1 \pmod 4$ we write $[ \beta_p]$ to denote the class $[ \beta ] \in W(\mathbb F_p)$ where $\beta$ is any non-square in $\mathbb F_p^\ast$. We construct explicit preimages under $\partial :W(\mathbb Q) \to \oplus _{p \text{ prime}} W(\mathbb F_p)$ of $[ 1_p]$ and $[ \beta_p]$, and since these elements additively generate each $W(\mathbb F_p)$, we obtain preimages for all elements in $\oplus _{p \text{ prime}} W(\mathbb F_p)$.  
\begin{proposition}  \label{PropsitionPreimagesOfDel}
Let $p$ be a prime, then 
\begin{itemize}
\item[(i)] $[ \pm p ] \in \partial ^{-1}([\pm 1_p]).$
\item[(ii)] $[ p a ] \oplus [ a ] \in \partial ^{-1}([ \beta_p] ).$
\end{itemize}
In case (ii) we assume that $p\equiv 1 \pmod 4$ and that $a$ is a prime with $a\equiv 3 \pmod 4$ and $a\equiv \beta_p \pmod p$. Such primes $a$ always exist. 
\end{proposition} 
\begin{proof}
Part (i) of the proposition is self-evident. Turning to part (ii), assume $p$ to be a prime with $p\equiv 1 \pmod 4$. By the Chinese Remainder Theorem there exists a solution $x$ to the Diophantine system 
\begin{align*}
x & \equiv 3 \pmod 4 \cr
x & \equiv \beta_p \pmod p.  
\end{align*}
The set $\{ x+ 4pk\,|\, k\in \mathbb N\}$ contains infinitely many primes by Dirichlet's Theorem on primes in arithmetic progressions. Let $a$ be one such prime, then $a= x+4pk$ for some $k\in \mathbb N$. Clearly $a\equiv x \pmod 4 \equiv 3 \pmod 4$ and $a\equiv x \pmod p \equiv \beta_p \pmod p$.  By Hilbert's quadratic reciprocity theorem, since $a$ is congruent to the non-square $\beta_p$ modulo $p$, then $p$ is also congruent to a non-square modulo $a$ (since $p\equiv 1\pmod 4$) and thus congruent to $-1$ modulo $a$ since $a\equiv 3 \pmod 4$. Therefore
\begin{align*}
\partial _p( [ pa] \oplus [ a ]) & = [ a ]  = [ \beta_p ] \in W(\mathbb F_p), \cr
\partial _a( [ pa] \oplus [ a ]) & = [ p ]  \oplus [ 1 ] = [ -1 ]  \oplus  [ 1 ]  = 0\in W(\mathbb F_a).
\end{align*} 
Clearly $\partial _b([ pa] \oplus [ a ])=0$ for any prime $b\ne p, a$. This proves the proposition. It is worth noting that $\partial^{-1}([ \beta_p])$ does not contain a rank 1 form (easy exercise), showing that $ws([ \pm 1_p]) = 1$ and $ws([ \beta_p])=2$. 
\end{proof}
\begin{example} \label{ExampleCalculatingWittSpan}
Let $B$ be the rank 10 and signature 0 form 
$$B=\langle -35 \rangle \oplus \langle -21 \rangle \oplus\langle -20 \rangle \oplus\langle -15 \rangle \oplus\langle -3 \rangle \oplus\langle 55 \rangle \oplus\langle 77  \rangle \oplus\langle 78  \rangle \oplus\langle 98 \rangle \oplus\langle 150 \rangle.$$
It is easy to calculate that  $\partial [B] = [ 1_3] \oplus  [ -1_{7}] \oplus  [ \beta_{13}] $. Using Proposition \ref{PropsitionPreimagesOfDel} we can construct the rank 4 and signature 2 form $B'$ with $\partial [B'] = \partial [B]$ by letting $B_2$ be the sum of the preimages of each of $[ 1_3]$, $[ -1_7]$ and $[ \beta_{13}]$ as specified in said proposition. This produces the form 
$$B' = \langle -7 \rangle \oplus \langle 3 \rangle  \oplus \langle 67 \rangle \oplus \langle 871 \rangle.$$ 
In this construction we picked $\beta _{13} = 2$ and the corresponding prime $a$ from Proposition \ref{PropsitionPreimagesOfDel} to be $a=67$, note that $871 = 13\cdot 67$.  The form $B'$ itself is isotropic and 
$$B' \cong B_ 2\oplus \mathbb H \qquad \text{ with } \qquad B_2= \langle 871  \rangle \oplus\langle 1407  \rangle.$$
The forms $B_2\oplus \langle -1\rangle $ and $B_2\oplus \langle -1\rangle \oplus \langle -1\rangle $ are both anisotropic over $\mathbb Q$ (indeed $B_2\oplus \langle -1\rangle \oplus \langle -1\rangle $ is anisotropic over $\mathbb Q_2$) and $B_2\oplus \langle -1\rangle \oplus \langle -1\rangle  \oplus \langle -1\rangle \cong B_{-1}\oplus \mathbb H$, with 
$$B_{-1} = \langle -17\, 863\, 741 \rangle \oplus \langle -313\, 726 \rangle\oplus \langle 31\, 224\, 068\, 621\, 382 \rangle .$$
The form $B_{-1}\oplus \langle -1 \rangle $ is anisotropic over $\mathbb Q_2$ by Part (ii) of Theorem \ref{TheoremCharacterizingSmallRankFormsOverQp}, and thus also over $\mathbb Q$. It follows then that $B_{-1}\oplus \langle -1 \rangle \oplus \langle -1 \rangle \cong B_{-3}\oplus \mathbb H$ for a rank 3 anisotropic form $B_{-3}$ with signature $-3$. The precise nature of this form is irrelevant because $B_{-3}\oplus \langle -1\rangle $ is negative definite and so anisotropic, while $B_{-3}\oplus \langle 1\rangle $ is isomorphic to $B_{-1}\oplus \langle -1 \rangle$ and thus also anisotropic. We are now in a position to apply Theorem \ref{TheoremAlgorithmForComputingws} to find that $ws([B]) = 2$. Indeed in the  notation of said theorem, $\mathcal I_2=\{2\}$ and the case considered here corresponds to Case (b2) from the proof of Theorem \ref{TheoremAlgorithmForComputingws}. 
\end{example}
\subsection{Proof  of Theorems \ref{Consequence1} and  \ref{TheoremCharacterizingWsTwoThree}}
\subsubsection{Proof of Theorem \ref{Consequence1}} 
One direction of the theorem is clear: If $W(K) = [d]\oplus (n\cdot [1])$ regardless of the value of $d$, then $ws(K) \le 1$. Moreover $ws(K)=1$ if and only if $d$ is not a square. 

In the other direction, let $K$ be a knot with $ws(K) \le 1$. Then the rational Witt class $W(K)$ of the knot $K$ is represented by a symmetric, bilinear form of the type 
\begin{equation}\label{EquationFirstRepresentative}
\langle d \rangle \oplus (n \cdot \langle 1 \rangle)
\end{equation}
with $n=\sigma (K) - \text{Sign}(d)$.  On the other hand, $W(K)$ is also represented by the symmetrized linking form obtained from any Seifert surface for $K$. If we diagonalize this form using the Gram-Schmidt process (discarding hyperbolic summands, if any), we find that $W(K)$ is represented by a form of the type
\begin{equation}\label{EquationSecondRepresentative}
\langle a_1\rangle \oplus \dots \oplus \langle a_{2g}\rangle \qquad \text{ with } \qquad | a_1 \cdot \dots \cdot a_{2g} | = \det K.
\end{equation}
The isomorphism types of the two forms \eqref{EquationFirstRepresentative} and \eqref{EquationSecondRepresentative} are related by a finite number of moves of these two types (\cite{MilnorHusemoller}, Page 85):
\begin{itemize}
\item $\langle a \rangle \oplus \langle -a \rangle  = 0$. 
\item$\langle a \rangle \oplus \langle b \rangle  = \langle a+b\rangle \oplus \langle ab/(a+b)\rangle$ whenever $a+b\ne 0$. 
\end{itemize}
These moves preserve the determinant of the form up to sign, showing that $d = \pm \det K$.

\subsubsection{Proof of Theorem \ref{TheoremCharacterizingWsTwoThree}}
Let $B$ be a symmetric bilinear form. Suppose there exists a prime $p\equiv 1 \pmod 4$ such that $\partial _p[B] = [\beta_p]$ and $\partial _q [B] = 0$ for all primes $q\ne p$. By Proposition \ref{PropsitionPreimagesOfDel}, the form $A_2 = \langle pa \rangle \oplus \langle a \rangle$ has the property that $\partial [A_2] = \partial [B]$. Recall that $a$ here is a prime congruent to $3 \pmod 4$ and also congruent to $\beta_p \pmod p$.  Let  $A_{-2} = \langle -pa \rangle \oplus \langle -a \rangle$, then 
\begin{align*}
\partial _p [A_{-2}] & = [-a] = [-\beta_p] = [\beta_p]\in W(\mathbb F_p), \cr
\partial _a [A_{-2}] & = [-p]\oplus [-1] = [1]\oplus [-1] = 0 \in W(\mathbb F_a). 
\end{align*}
We see that $\partial [A_{-2}]= \partial [A_2] = \partial [B]$, and each of the four forms $A_{\pm 2}\oplus \langle \pm 1\rangle$ is anisotropic over $\mathbb Q$ (as can be seen by evaluating both residue homomorphims $\partial ^1$ and $\partial ^2_p$ from \eqref{EquationTheResidueHomomorphisms} on these forms: $\partial ^1([A_{\pm 2}]\oplus [ \pm 1]) =  [1_p]\oplus [\beta_p]$ and $\partial ^2_p ([A_{\pm 2}]\oplus [ \pm 1]) = [\beta_p]$). 

Both of the forms $A_2 \oplus \langle -1 \rangle \oplus \langle -1 \rangle$ and $A_{-2} \oplus \langle 1 \rangle \oplus \langle 1 \rangle$ are isotropic over $\mathbb R$ and over $\mathbb Q_p$ for every odd prime $p$. Additionally, both forms are anisotropic over $\mathbb Q_2$ if $p\equiv 1 \pmod 8$, and both isotropic over $\mathbb Q_2$ if $p\equiv 5 \pmod 8$ (cf. Part (iii) of Theorem \ref{TheoremCharacterizingSmallRankFormsOverQp}). Thus, if $p\equiv 1 \pmod 8$ then $\mathcal I_2=\{\pm 2\}$ with $A_i \oplus \langle \pm 1\rangle$ anisotropic for each $i\in I_2$, while if $p\equiv 5 \pmod 8$ then 
$$A_2\oplus \langle -1\rangle \oplus \langle -1\rangle \cong A_0 \oplus \mathbb H \cong A_{-2}\oplus \langle 1\rangle \oplus \langle 1\rangle .$$
In this case $\mathcal I_2=\{-2,0,2\}$ and since $A_0 \oplus \langle -1\rangle \cong A_{-2}\oplus \langle 1 \rangle$ and $A_0 \oplus \langle 1\rangle \cong A_{2}\oplus \langle -1 \rangle$, we see again that $A_i \oplus \langle \pm 1\rangle$ is anisotropic for each $i\in \mathcal I_2$. Theorem \ref{TheoremAlgorithmForComputingws}, Part (iii) now guarantees that $ws([B]) = 2$. This proves Part (i) of Theorem \ref{TheoremCharacterizingWsTwoThree}. 
\vskip3mm
Next suppose that $B$ is a form with $\partial _p[B] = [1_p]\oplus [\beta_p]$ and $\partial _q [B] = 0$ for all primes $q\ne p$. By Proposition \ref{PropsitionPreimagesOfDel}, the form $A_3 = \langle p \rangle \oplus \langle pa \rangle \oplus \langle a \rangle$ has the property that $\partial [A_3] = \partial [B]$ (as before, $a$ is a prime congruent to $3 \pmod 4$ and congruent to $\beta_p \pmod p$).  We define three additional forms: 
\begin{align*}
A_{1} & = \langle -p \rangle \oplus \langle pa \rangle \oplus \langle a \rangle, \cr
A_{-1} & = \langle p \rangle \oplus \langle -pa \rangle \oplus \langle -a \rangle, \cr
A_{-3} & = \langle -p \rangle \oplus \langle -pa \rangle \oplus \langle -a \rangle.
\end{align*}
An explicit computation shows that $\partial [A_i] = \partial[B]$ for each $i\in \{\pm 1, \pm 3\}$. Moreover, all forms $A_i \oplus \langle \pm 1 \rangle $ are anisotropic since the residue homomorphisms applied at the prime $p$ give:
\begin{align*}
\partial^1 ([A_i] \oplus [ \pm 1 ]) & = [a]\oplus [\pm 1] = [\beta_p]\oplus [1] \in W(\mathbb F_p), \cr
\partial_p^2 ([A_i] \oplus [ \pm 1 ]) & = [\pm 1_p]\oplus [\pm a] = [1_p]\oplus [\beta_p] \in W(\mathbb F_p)
\end{align*}
Thus $\mathcal I_3=\{\pm 1, \pm 3\}$ $A_i \oplus \langle \pm 1 \rangle $ is anisotropic for all $i\in \mathcal I_3$. Theorem \ref{TheoremAlgorithmForComputingws}, Part (iv) shows that $ws([B]) = 3$, completing the proof of Theorem \ref{TheoremCharacterizingWsTwoThree}. 
\bibliographystyle{plain}
\bibliography{bibliography.bib}
\end{document}